\documentclass[10pt,reqno,final]{amsart}
	\usepackage{graphicx}
\usepackage{boxedminipage}
\usepackage{amsmath}
\usepackage{epsf}
 \usepackage{epstopdf}
\usepackage{latexsym}
\usepackage{amssymb}
\usepackage{amsthm}
\usepackage{extarrows}
\usepackage{color}
\usepackage[english]{babel}
\usepackage{mathrsfs}
      \usepackage{color}
\usepackage[usenames,dvipsnames]{xcolor}
\usepackage{afterpage}

\newtheorem{theorem}{Theorem}
\newtheorem{remark}{Remark}

\newtheorem{lemma}{Lemma}
\newtheorem{example}{Example}

\begin{document}
\author{Can Huang}
 \address[Can Huang]{School of Mathematical Science, Xiamen University, Fujian, 361005, China.}
 \email{canhuang@xmu.edu.cn}
\author{Qingshuo Song}
\address[Qingshuo Song]{Department of Mathematics, City University of Hong Kong, Kowloon, Hong Kong.}
\author{Zhimin Zhang}
\address[Zhimin Zhang]{Beijing Computational Research Center, Haidian, Beijing, 100084, China and Department of Mathematics, Wayne
State University, Detroit, MI 48202, USA.}

\keywords{substantial fractional differential equation, spectral method, collocation method, Petrov-Galerkin, generalized Laguerre polynomials, condition number}
\subjclass{65N35, 65L60, 65M70,  41A05, 41A30, 41A25}
\title{\bf{Spectral Method for Substantial Fractional Differential Equations}}
\maketitle
\begin{abstract}
In this paper, a non-polynomial spectral Petrov-Galerkin method and  associated collocation method for substantial fractional differential equations (FDEs) are proposed, analyzed, and tested. We extend a class of generalized Laguerre polynomials  to form our basis. By a proper scaling of trial basis and test basis, our Petrov-Galerkin method results in a diagonal and thus well-conditioned linear systems for both fractional advection equation and fractional diffusion equation. In the meantime, we construct substantial fractional differential collocation matrices and provide explicit forms for both type of equations. Moreover,  the proposed method allows us to adjust a parameter in basis selection according to different given data to maximize the convergence rate. This fact has been proved in our error analysis and confirmed in our numerical experiments.
\end{abstract}

\section{Introduction}
With the development of advanced experimental technology, more and more particle diffusion process in a complex system are revealed to follow anonymous diffusion instead of the traditional Gaussian statistics. It is characterized by deviations from traditional linear time dependence in its second moment $<X(t)>\sim t^\alpha, \alpha\neq 1$. In particular, subdiffusion process ($0<\alpha<1$) has been a focal point in both physics and mathematics by virtue of its internal nature of connection with fractional differential equations (FDEs). This class of process has either infinite mean of waiting time or diverging jump length variance (L\'{e}vy flights). As a result, continuous time random walk (CTRW), rather than Brownian motion, seems to be  more competitive to model the anonymous sub-diffusion of particles in a complex system and this fact has been adopted in many applications such as underground environment problem \cite{HH98}, fluid flow \cite{KSZ96}, and turbulence and chaos \cite{Shle86}. Regarding to probability density function,  the CTRW  with diverging mean of waiting time results in a Fokker-Planck equation (FPE) of fractional derivative in time whereas L\'{e}vy flights  leads to a FPE of spatial fractional derivative \cite{MK00}.

Based upon an extension of CTRW to position-velocity space, Fredrichs et al \cite{FJBE06} generalized the concept of fractional derivative to a substantial one as follows
\begin{equation}\label{def1}
D_s^{\nu} f(x)=\frac{1}{\Gamma(\nu)}\bigg[\frac{d}{dx}+\sigma\bigg]\int_0^x (x-\tau)^{\nu-1}e^{-\sigma (x-\tau)}f(\tau)d\tau,\ 0<\nu<1,
\end{equation}
which has been investigated in \cite{CB11, CTB10} and references therein. It is noteworthy that the definition was extended to any order of $\nu>0$ by Deng recently \cite{CD}. In this work, we will concentrated on non-polynomial spectral Petrov-Galerkin method  associated collocation method for substantial fractional differential equations of order $\nu$ and $(1+\nu)$ with $0<\nu<1$.

By allowing trial space and test space to be different, the Petrov-Galerkin method has a remarkable edge over the Galerkin method on the choice of test space to enhance computational efficiency while preserving its convergence order. A general framework of the Petrov-Galerkin method for second type integral equation has been established in \cite{XuYS}.  By a careful selection of these two spaces, Karniadakis et al \cite{ZK14JCP} obtained an explicit expression of their approximation (without solving a linear system) for certain Riemann-Liouville FDEs. For substantial FDEs, our work seems to be the first attempt by this method.

Spectral collocation method  approximates true solution of an equation at a certain set of prescribed collocation points by multiplication of a spectral collocation matrix on a vector of unknowns. For standard integer order problems, interested readers are referred to \cite{CHQZ06,STW11,Tre00}.  Spectral collocation matrices for Riemann-Liouville fractional derivative are first proposed by Karniadakis et al \cite{ZK14} on the basis of ``poly-factonomial'' approximation, which is of the form $\{(x+1)^\mu P_n^{\alpha,\beta}(x)\}_{n=0}^N$, where $P_n^{\alpha,\beta}(x)$ is the Jacobi polynomial.  Akin to the collocation matrices for integer order derivatives, the condition number of the  Gauss-Lobatto collocation matrices for fractional differential equations grow like $\mathcal{O}(N^{2\nu})$, where $\nu$ is the order of the fractional derivative.  In order to circumvent the difficulty, Huang et al proposed well-conditioned collocation methods for standard fractional differential equations  based upon the Birkhoff interpolation \cite{HW14}. For finite difference methods, spectral Galerkin methods and finite element methods for FDEs, readers are referred to \cite{Diet1, JLP, Pod, Tadj, Xu1, Xu2, ZK14JCP}.

Compared to the extensive numerical methods developed for standard FDE, the effort on developing numerical schemes for substantial FDE is limited accounting for its relatively new appearance in the field. As far as we know,  numerical methods for substantial FDE are prominently finite difference methods (FDM), see \cite{BM10} for first order accuracy FDM and particle tracking methods. In \cite{MM09}, a comparison study of numerical solutions of three fractional partial differential equations falling in the class of L\'{e}vy models with substantial fractional derivative is explored. Recently, high order finite difference schemes, namely,  the tempered weighted and shifted Gr\"{u}nwald difference, has received some interest \cite{LD}. In contrast, spectral methods for substantial FDE is relatively scare and this work is our first exploration in this direction.  Our essential idea is
to find suitable basis for each specific equation which incorporates initial (boundary) conditions automatically and has explicit expression of
required substantial fractional derivative. Henceforth, our method has the following prominent features:

 (1) Our basis consists of a combination of generalized Laguerre polynomials $L_n^\alpha(x)$, exponential function $e^{-\sigma x}$, and power function $x^\alpha$. Hence, our method is far from a polynomial approximation. We further extend generalized Laguerre polynomials $L_n^\alpha(x)$ from $\alpha>-1$ to $\alpha\leq -1$. One distinct feature of our extension is that the obtained result is again a polynomial for all real $\alpha$, which is essentially different from the extension explored in \cite{G13}.

(2) By a careful selection of trial space and test space, our Petrov-Galerkin method yields diagonal and well-conditioned linear system for our  model problems.

 (3) In view of the priori known regularity of given data (usually at $t=0$), we are able to adjust the parameter $\alpha$ in our basis to achieve high order accuracy. Note that $\alpha=1-\nu$ indicates smooth function approximation. Therefore, convergence rate can be enhanced for some specific choice of  parameters.

 (4) To the best of our knowledge, this work is the first attempt to solve a substantial FDE on a semi-infinite domain without a domain truncation so far.

We point out that for standard Riemann-Liouville FDE, Petrov-Galerkin methods and a spectral collocation method have been proposed in \cite{CSW14, ZK14JCP, ZK14}. However, our method is far from a trivial extension of theirs. Firstly, the domain of substantial FDE is extended from $[-1,1]$ to $[0,\infty)$, which brings more challenging initial/boundary value conditions. Thus, our basis is distinct from either ``poly-factonomial'' \cite{ZK14JCP} or ``Generalized Jacobi Function'' \cite{CSW14} and is more complicated; Secondly, our spectral collocation algorithms are completely explicit. We provide a closed form of conversion from Lagrange interpolation polynomial to generalized Laguerre polynomials $L_n^\lambda (x)$.

The rest of the paper is organized as follows. In section 2, we shall recall some preliminary knowledge on modified Laguerre polynomials and make a necessary further extension. Some essential identities pertinent to our algorithms shall be introduced.  Then in section 3, we shall
explore Petrov-Galerkin method for substantial FDEs. Convergence analysis and numerical experiments will be provided.
In section 4, we shall present our fractional Lagrange interpolant, which satisfy the Kronecker delta property at collocation points and initial/boundary conditions. Based upon it, explicit spectral collocation algorithms and their convergence analysis will be presented. Our theoretical results are confirmed by associated numerical experiments.

\section{Preliminary}
In this section, we briefly present some preliminary knowledge pertaining to our sub-sequential algorithms. We begin with some notations and definitions in associated fractional calculus. In this work, we adopt the definition of substantial fractional derivative in \cite{CD} (note that a slight difference exists between the definition in \cite{FJBE06} and that in \cite{CD} for $0<\nu<1$).

Let $m$ be the smallest integer that exceeds $\nu$. Then
\begin{equation}\label{def2}
D_s^\nu f(x)=\frac{1}{\Gamma(m-\nu)}D_s^m\bigg[\int_0^x (x-\tau)^{m-\nu-1}e^{-\sigma (x-\tau)}f(\tau)d\tau\bigg],
\end{equation}
where
\begin{equation*}
D_s^m=\bigg(\frac{d}{dx}+\sigma\bigg)^m=\underbrace{(D+\sigma)\cdots (D+\sigma)}_{m\  times}.
\end{equation*}
\subsection{Modified Laguerre Function (MLF)}
Associated with the definition of substantial fractional derivative \eqref{def2}, we define MLF by
\begin{equation}
\hat{L}_n^{\lambda}(x)=x^\lambda e^{-\sigma x}L_n^{\lambda}(2\sigma x),
\end{equation}
where $0<\lambda<1$, $\sigma$ is inherited from \eqref{def2} and $L_n^{\lambda}(y)$ is the standard generalized Laguerre polynomials. It is noteworthy that the generalized Laguerre function defined in \cite[Page 241]{STW11} is a special case of our definition. Clearly, $\hat{L}_n^{\lambda}(x)$ is orthogonal with respect to the weight $\hat{w}^{\lambda}(x)=x^{-\lambda}$.

We next introduce hypergeometric functions for future use.  A general hypergeometric series with $p$ upper parameters and $q$ lower
parameters is defined as follows:
\begin{equation*}
_pF_q(a_1,\cdots, a_p; b_1,\cdots, b_q; z)=\sum\limits_{k=0}^\infty \frac{(a_1)^{(k)}\cdots (a_p)^{(k)}}{(b_1)^{(k)}\cdots (b_q)^{(k)}}\frac{z^k}{k!},
\end{equation*}
where $(a)^{(k)}$ is the Pochhammer symbol $(a)^{(k)}=(a)(a+1)^{(k-1)}$.

Some useful properties of MLF or standard Laguerre polynomials are listed as follows.

$\bullet$ Three-term recurrence relation:
  \begin{equation}
 \left\{ \begin{array}{l}
 \hat{L}_0^\lambda (x)=x^{\lambda}e^{-\sigma x}, \hat{L}_1^\lambda(x)=x^\lambda e^{-\sigma x}(-2\sigma x+\lambda+1),\\
 (n+1)\hat{L}_{n+1}^\lambda(x)=(2n+\lambda+1-2\sigma x)\hat{L}_n^\lambda(x)-(n+\lambda)\hat{L}_{n-1}^\lambda(x).
 \end{array}
 \right.
  \end{equation}
It provides a stable approach to evaluate MLFs in our emerging algorithms. Furthermore,
$$\hat{L}_n^\lambda(x)\to 0\ \text{as}\ x\to \infty.$$

$\bullet$ Orthogonality:
\begin{equation}
 \int_0^\infty  x^{-\lambda}\hat{L}_n^\lambda(x)\hat{L}_m^\lambda(x)dx=\bigg(\frac{1}{2\sigma}\bigg)^{1+\lambda} \frac{\Gamma(n+\lambda+1)}{n!}\delta_{nm}.
  \end{equation}

 $\bullet$ Differential formulas \cite{STW11}:
   \begin{subequations}\label{diff1}
   \begin{align}
  & \frac{d^k}{dx^k}L_n^\lambda(x)=(-1)^kL_{n-k}^{\lambda+k}(x),\\
  &xL_{n-1}^{\lambda+1}(x)=(n+\lambda)L_{n-1}^{\lambda}(x)-nL_n^\lambda(x),\\
   &L_n^{\lambda}(x)=L_n^{\lambda+1}(x)-L_{n-1}^{\lambda+1}(x).
   \end{align}
   \end{subequations}
   A simple computation yields
   \begin{equation}\label{diff2}
   x^2L_{n-2}^{\lambda+2}(x)=(n+\lambda)(n-1+\lambda)L_{n-2}^\lambda(x)-2(n-1)(n+\lambda)L_{n-1}^\lambda(x)+n(n-1)L_{n}^\lambda(x).
   \end{equation}

$\bullet$ Essential identities \cite{PBM86}:
 \begin{eqnarray}\label{originalI}
&& \int_0^a (a-x)^{\beta-1} x^{\alpha-1} L_n^\lambda (cx)dx=\frac{a^{\alpha+\beta-1}(\lambda+1)_n}{n!} B(\alpha,\beta) \ _2F_2(-n,\alpha; \alpha+\beta,\lambda+1; ac), \alpha,\beta>0,\quad
 \end{eqnarray}
 where $B(\cdot,\cdot)$ is the beta function.  In particular, if $\alpha=\lambda+1$
 \begin{equation}
 \int_0^a (a-x)^{\beta-1} x^\lambda L_n^\lambda (cx)dx=a^{\beta+\lambda} B(\beta,\lambda+n+1)L_n^{\beta+\lambda}(ac).
 \end{equation}
This identity immediately implies
\begin{equation}
\int_0^x (x-\tau)^{\nu-1}e^{-\sigma(x-\tau)}\hat{L}_n^\lambda (\tau)d\tau=B(\nu,\lambda+n+1)\hat{L}_n^{\nu+\lambda}(x).
\end{equation}
Henceforth, by \eqref{diff1}, a simple computation implies
\begin{eqnarray}
D_s^{1-\nu} [\hat{L}_n^\lambda (x)]&=& \frac{\Gamma(\lambda+n+1)}{\Gamma(\nu+\lambda+n+1)x} \bigg[(\nu+\lambda)\hat{L}_n^{\lambda+\nu}(x)-2\sigma\hat{L}_{n-1}^{\lambda+\nu+1}(x)\bigg],\nonumber\\
&=&\frac{\Gamma(\lambda+n+1)}{\Gamma(\nu+\lambda+n+1)}x^{\lambda+\nu-1}e^{-\sigma x}(n+\lambda+\nu)
\bigg(L_n^{\lambda+\nu}(2\sigma x)-L_{n-1}^{\lambda+\nu}(2\sigma x)\bigg)\nonumber\\
&=&\frac{\Gamma(\lambda+n+1)}{\Gamma(\nu+\lambda+n)}\hat{L}_n^{\lambda+\nu-1}(x). \label{essentialI1}\\
D_s^{2-\nu} [\hat{L}_n^\lambda (x)]&=& \frac{\Gamma(\lambda+n+1)}{\Gamma(\nu+\lambda+n+1)x^2} \bigg[(\nu+\lambda)_2\hat{L}_n^{\lambda+\nu}(x)-4\sigma(\lambda+\nu)\hat{L}_{n-1}^{\lambda+\nu+1}(x)+4\sigma^2\hat{L}_{n-2}^{\lambda+\nu+2}(x)\bigg]\nonumber\\
&=& \frac{\Gamma(n+\lambda+1)}{\Gamma(n-1+\lambda+\nu)}\hat{L}_n^{\lambda+\nu-2}(x). \label{essentialI2}
\end{eqnarray}
where $(a)_k=a(a-1)\cdots (a-k+1)$ is the falling factorial.

Identities \eqref{essentialI1} and \eqref{essentialI2} allow for both smooth function approximation ($\lambda=0$) and function with a certain weak singularity approximation for a substantial fractional differential equation.  If the index $\lambda+\nu-2\leq -1$ of $\hat{L}_{n}^{\lambda+\nu-2}(x)$, we shall need to extend the Laguerre polynomials further.

\subsection{Extended Laguerre polynomial}
We extend the Laguerre polynomial based upon a formula \cite{Szego}
\begin{equation}
L_n^\alpha(x)=L_n^{\alpha+1}(x)-L_{n-1}^{\alpha+1}(x), \quad L_0^\alpha(x)=1.
\end{equation}
Unlike the extension for Jacobi polynomial in \cite{GSW09}, our extension result is always a polynomial. Since $\lambda+\nu-2>-2$, we thereby restrict our attention to the case $-2<\alpha\leq-1$. However, we can easily extend $\alpha\in R$ by the same fashion. It is noteworthy that in  \cite{Szego} Szeg\"{o} suggests a possible extension by virtue of hypergeometric form of Laguerre polynomials. Specifically, when $\alpha=-k$ for $k\geq 1$ and $k\in \mathcal{N}$,
\begin{equation*}
L_n^{-k}=(-x)^k \frac{(n-k)!}{n!}L_{n-k}^k(x).
\end{equation*}
Unfortunately, the exploration of properties of the extended polynomial is limited. We find that our extended Laguerre polynomials preserves basic properties of standard Laguerre polynomials as follows. See appendix A for detailed verifications.

$\bullet$ The Sturn-Liouville equation:
\begin{equation}
\partial_x[e^{-x}x^{\alpha+1}\partial_x L_n^\alpha(x)]+nx^\alpha e^{-x}L_n^\alpha(x)=0.
\end{equation}

$\bullet$ Recursive relation:
\begin{equation}
(n+1)L_{n+1}^\alpha(x)=(2n+1+\alpha-x)L_n^\alpha(x)-(n+\alpha)L_{n-1}^\alpha(x).
\end{equation}

$\bullet$ Orthogonality:
\begin{equation}\label{ex:der}
\int_0^\infty \partial_x^m L_n^\alpha(x)\partial_x^m L_k^\alpha(x)x^{\alpha+m}e^{-x}dx=\frac{\Gamma(n+\alpha+1)}{\Gamma(n-m+1)}\delta_{n,k}:=\gamma_{n,m}\delta_{n,k}, m\geq 0.
\end{equation}

For the sake of analysis, we introduce a space as \cite{GSW09}
\begin{equation}
B^r_{w^\alpha}([0,\infty))=\{u\ \text{is\ measurable\ and}\ \|u\|_{r,w^\alpha}<\infty\}
\end{equation}
equipped with norm
\begin{equation}
\|u\|_{r,w^\alpha}=\bigg(\sum\limits_{k=0}^r\|\partial_x^ku\|_{w^{\alpha+k}}^2\bigg)^{1/2}
\end{equation}
and the weight $w^\alpha(x)=x^\alpha e^{-x}$. Consider the orthogonal projection $\Pi_N^\alpha: B_{w^\alpha}^r\to P_N$
such that
\begin{equation}
(u-\Pi_N^\alpha u, v)=0, \ \forall v\in P_N.
\end{equation}

\begin{lemma}\label{Newlem1}
For any $u\in B^r_{w^\alpha}$,
\begin{equation}
\|\partial_x^m(u-\Pi_N^\alpha u)\|_{w^{\alpha+m}}\leq CN^{(m-r)/2} \|\partial_x^ru\|_{w^{\alpha+r}}.
\end{equation}
\end{lemma}
\begin{proof}
It is clear that
\begin{equation}
u-\Pi_N^\alpha u=\sum\limits_{n=N+1}^\infty \hat{u}_n L_n^\alpha(x).
\end{equation}
The case $\alpha+m>-1$ has been proved in \cite{STW11}.  For $-2<\alpha\leq -1$, we only need to consider the case $m=0$, otherwise,
the index turn back to the standard case.
\begin{eqnarray}
\| (u-\Pi_N^\alpha u)\|^2_{w^{\alpha}}&=&\sum\limits_{n=N+1}^\infty \hat{u}_n^2\gamma_{n,0}=\sum\limits_{n=N+1}^\infty \hat{u}_n^2\gamma_{n,r}\frac{\gamma_{n,0}}{\gamma_{n,r}}\nonumber\\
&\leq& \frac{\gamma_{N+1,0}}{\gamma_{N+1,r}} \|\partial_x^ru\|_{w^{\alpha+r}}\nonumber\\
&\leq& CN^{-r} \|\partial_x^ru\|_{w^{\alpha+r}},
\end{eqnarray}
where, $\gamma_{n,k}$ is defined in \eqref{ex:der}.
\end{proof}
\subsection{Laguerre-Gauss quadrature and Laguerre-Gauss-Radau quadrature}
In our algorithms for spectral collocation method, we shall use these two types of numerical quadrature, respectively.
Let $\{x_i,w_i\}_{i=0}^N$ be the Laguerre-Gauss quadrature or Laguerre-Gauss-Radau points and weights associated with $\omega(x)=x^\lambda e^{-2\sigma x}$. Then \cite{STW11},

$\bullet$ For the Laguerre-Gauss quadrature,
\begin{eqnarray}
&& \{x_i\}_{i=0}^N\ \text{are zeros of } \ L_{N+1}^{\lambda}(2\sigma x), \notag\\
&&w_i=\frac{\Gamma(N+\lambda+1)}{(N+1)!(N+\lambda+1)}\frac{x_i^\lambda}{[L_N^\lambda (2\sigma x_i)]^2}\bigg(\frac{1}{2\sigma}\bigg)^{1+\lambda}, 0\leq i\leq N.
\end{eqnarray}

$\bullet$ For the Laguerre-Gauss-Radau quadrature,
\begin{eqnarray}
&&x_0=0, \ \text{and}\ \{x_i\}_{i=1}^N\ \text{are zeros of } \ L_N^{1+\lambda}(2\sigma x).\notag\\
&&w_0=\frac{(\lambda+1)\Gamma^2(\lambda+1)N!}{\Gamma(N+\lambda+2)}\bigg(\frac{1}{2\sigma}\bigg)^{1+\lambda}\notag\\
&&w_i=\frac{\Gamma(N+\lambda+1)}{N!(N+\lambda+1)}\frac{1}{[L_N^\lambda (2\sigma x_i)]^2}\bigg(\frac{1}{2\sigma}\bigg)^{1+\lambda}, 1\leq i\leq N.
\end{eqnarray}

For both sets of quadrature points and weights, we have
\begin{equation}
\int_0^\infty p(x)x^\lambda e^{-\sigma x}dx=\sum\limits_{i=0}^N w_ip(x_i), \forall p\in P_{2N}.
\end{equation}

\section{Petrov-Galerkin method}
Motivated by \eqref{essentialI1} and \eqref{essentialI2}, we establish our variational form in the trial space $U_1=\text{span} \{(n+1)^{-\frac{(\lambda-\nu+1)}{2}}\hat{L}^{\lambda}_n(x)\}$ and test space $V_1=\text{span}\{(n+1)^{-\frac{(\lambda+\nu-1)}{2}}\hat{L}_n^{\lambda+\nu-1}(x)\}$ with weight $\hat{w}_1(x)=x^{1-\lambda-\nu}$ for fractional advection equation,  and $U_2=\text{span}\{(n+1)^{-\frac{(\lambda-\nu+2)}{2}}\hat{L}^{\lambda}_n(x)\}$ and  $V_2=\text{span}\{(n+1)^{-\frac{(\lambda+\nu-2)}{2}}\hat{L}_n^{\lambda+\nu-2}(x)\}$ with weight $\hat{w}_2(x)=x^{2-\lambda-\nu}$ for fractional diffusion equation.
\subsection{Substantial fractional advection equation}
We seek the approximation of solution to the simplest substantial FODE
\begin{equation}\label{eqPG1}
\left\{
\begin{array}{l}
D_s^{1-\nu} u(x)=f(x), 0<\nu<1\\
\lim\limits_{x\to\infty}u(x)=0,
\end{array}
\right.
\end{equation}
which is of the form
\begin{equation}\label{PGuN}
u_N(x):=\sum\limits_{k=0}^N \frac{c_k}{(k+1)^{(\lambda-\nu+1)/2}} \hat{L}_k^{\lambda}(x),
\end{equation}
where $c_k$ are coefficients to be determined.  By projecting the \eqref{PGuN} into $V_1$, we obtain for $0\leq n\leq N$
\begin{eqnarray}
&&(f, \frac{1}{(1+n)^{(\lambda+\nu-1)/2}}\hat{L}_n^{\lambda+\nu-1})_{\hat{w}_1}\nonumber\\
&=&\sum\limits_{k=0}^N  \frac{c_k}{(1+n)^{(\lambda+\nu-1)/2}(k+1)^{(\lambda-\nu+1)/2}}\bigg(D_s^{1-\nu} \hat{L}_k^{\lambda}, \hat{L}_n^{\lambda+\nu-1}\bigg)_{\hat{w}_1}\nonumber\\
&=&\sum\limits_{k=0}^N  \frac{c_k}{(1+n)^{(\lambda+\nu-1)/2}(k+1)^{(\lambda-\nu+1)/2}}
\frac{\Gamma(k+\lambda+1)}{\Gamma(k+\nu+\lambda)}\bigg(\hat{L}_k^{\lambda+\nu-1},\hat{L}_n^{\lambda+\nu-1}\bigg)_{\hat{w}_1}\nonumber\\
&=&\sum\limits_{k=0}^N \frac{c_k}{(k+1)^\lambda}\frac{\Gamma(k+\lambda+1)}{\Gamma(k+1)}\bigg(\frac{1}{2\sigma}\bigg)^{\lambda+\nu}\delta_{k,n}\nonumber
\end{eqnarray}
Denote
\begin{eqnarray}
&&A=\bigg(\frac{1}{2\sigma}\bigg)^{\lambda+\nu}\begin{pmatrix}
\frac{\Gamma(\lambda+1)}{\Gamma(1)1^{\lambda}}&& & &\\
& \frac{\Gamma(\lambda+2)}{\Gamma(2)2^{\lambda}}&  & &\\
& &&\ddots & \\
& & & & \frac{\Gamma(\lambda+N+1)}{\Gamma(N+1)(N+1)^{\lambda}}
\end{pmatrix}\nonumber\\
&&F=\bigg(\int_0^\infty f(x)e^{-\sigma x}L_0^{\lambda+\nu-1}(2\sigma x)dx,\  \cdots, \frac{1}{(N+1)^{(\lambda+\nu-1)/2}}\int_0^\infty f(x)e^{-\sigma x}L_N^{\lambda+\nu-1}(2\sigma x)dx \bigg)^T.
\end{eqnarray}
The coefficients $C=(c_0, c_1,\cdots, c_N)^T$ is obtained by solving $AC=F$, where the integral in $F$ is computed by appropriate Gauss-Laguerre quadratures presented in subsection 2.3.

\begin{remark}
Note that our algorithm assumes the $u(0)=0$ for \eqref{eqPG1}.  If $u(0)=u_0\neq 0$, we can decompose the solution
$u(x)=w(x)+u_0e^{-\sigma x}$ and solve the following associated equation for $w(x)$:
\begin{equation}
D_s^{1-\nu} w(x)=f(x)-u_0\nu x^{\nu-1}e^{-\sigma x}, \quad
\lim\limits_{x\to\infty}w(x)=0.
\end{equation}
\end{remark}
\subsection{Substantial fractional diffusion equation}
Similarly, we consider equation
\begin{equation}\label{eqPG2}
\left\{
\begin{array}{l}
D_s^{2-\nu} u(x)=f(x), 0<\nu<1,\\
u(0)=0,\lim\limits_{x\to\infty}u(x)=0.
\end{array}
\right.
\end{equation}
The approximation $u_N(x)$ is of the form $u_N(x)=\sum\limits_{k=0}^N \frac{c_k}{(k+1)^{(\lambda-\nu+2)/2}} \hat{L}_k^\lambda(x)$. By equation \eqref{essentialI2} and orthogonality of the extended Laguerre polynomials,
\begin{eqnarray}
&&(f, \frac{1}{(n+1)^{(\lambda+\nu-2)/2}}\hat{L}_n^{\lambda+\nu-2})_{\hat{w}_2}\\
&=&\sum\limits_{k=0}^N \frac{c_k}{(n+1)^{(\lambda+\nu-2)/2}(k+1)^{(\lambda-\nu+2)/2}} \bigg(D_s^{2-\nu} \hat{L}_k^{\lambda}, {L}_n^{\lambda+\nu-2}\bigg)_{\hat{w}_2}\nonumber\\
&=&\sum\limits_{k=0}^N \frac{c_k}{(n+1)^{(\lambda+\nu-2)/2}(k+1)^{(\lambda-\nu+2)/2}} \frac{\Gamma(\lambda+k+1)}{\Gamma(\nu+\lambda+k-1)}\bigg(\hat{L}_k^{\lambda+\nu-2}, \hat{L}_n^{\lambda+\nu-2}\bigg)_{\hat{w}_2}\nonumber\\
&=&\sum\limits_{k=0}^N \frac{c_k}{(k+1)^{\lambda}} \frac{\Gamma(\lambda+k+1)}{\Gamma(k+1)}\bigg(\frac{1}{2\sigma}\bigg)^{\lambda+\nu-1}\delta_{k,n}
\end{eqnarray}

Hence, we solve a diagonal system $AC=F$, where
\begin{eqnarray}
&&A=\bigg(\frac{1}{2\sigma}\bigg)^{\lambda+\nu-1}\begin{pmatrix}
\frac{\Gamma(\lambda+1)}{\Gamma(1)1^{\lambda}}&& & &\\
& \frac{\Gamma(\lambda+2)}{\Gamma(2)2^{\lambda}}&  & &\\
& &&\ddots & \\
& & & & \frac{\Gamma(\lambda+N+1)}{\Gamma(N+1)(N+1)^{\lambda}}
\end{pmatrix}\nonumber\\
&&F=\bigg(\int_0^\infty f(x)e^{-\sigma x}L_0^{\lambda+\nu-2}(2\sigma x)dx,\  \cdots, \frac{1}{(N+1)^{(\lambda+\nu-2)/2}}\int_0^\infty f(x)e^{-\sigma x}L_N^{\lambda+\nu-2}(2\sigma x)dx \bigg)^T.
\end{eqnarray}

\begin{remark}
The effect of scaling factors  in test space or trial space is twofold.  Firstly, it plays the role of precondition factor for the matrix $A$; Secondly, it is indispensable for the verification of one essential condition in our convergence analysis.
\end{remark}

\begin{remark}
For an equation with non-homogeneous initial condition, we can take a similar process as that in Remark 1 to make a transformation.
\end{remark}
\subsection{Convergence analysis}
 Denote by $x^*(x):=<x,x^*>$ the action of bounded linear operator $x^*$ on $x$ in a Hilbert space $X$. Thus, by Rieze representation theory, $x^*(x)=(x,y) $  for some $y$ in the same space. Choose two finite-dimensional spaces $X_n$ and $Y_n$ satisfy condition (H)
 \cite{XuYS}:\ for any $x\in X$, we have $\|x_n-x\|\to 0$ and for any $y\in X^*$, we have $\|y_n-y\|\to 0$ as $n\to\infty$, and
 \begin{equation}
\text{dim}X_n=\text{dim} Y_n.
\end{equation}
 Furthermore, we call $\{X_n, Y_n\}$ a {\it regular pair} \cite{XuYS} if there exists a linear operator $\Pi_n: X_n\to Y_n$ with $\Pi_n X_n=Y_n$ and satisfy the conditions
\begin{eqnarray}
&&(a) \ \|x_n\|\leq C_1 <x_n,\Pi_n x_n>^{1/2}\ \text{for \ all}\  x_n\in X_n;\nonumber\\
&&(b) \ \|\Pi_nx_n\|\leq C_2 \|X_n\|\ \text{for \ all}\ x_n\in X_n,\nonumber
\end{eqnarray}
where $C_1$ and $C_2$ are independent of $n$.

 For any $x\in X$, $P_n x$ is a {\it generalized best approximation} from $X_n$ to $x$ with respect to $Y_n$ if we have the identity
 \begin{equation}
 (x-P_nx, y_n)=0, \ \text{for\ any\ } y_n\in Y_n.
 \end{equation}
Then, we have the  following two lemmas.
 \begin{lemma}[\cite{XuYS}]
 \label{lem1}
 For each $x\in X$, the generalized best approximation from $X_n$ to $x$ with respect to $Y_n$ exists uniquely if and only if
 $$Y_n\cap X_n^{\bot}=\{0\}.$$
 Under this condition, $P_n$ is a projection.
 \end{lemma}

 \begin{lemma}[\cite{XuYS}]
 \label{lem2}
 Assume that $\{X_n, Y_n\}$ satisfy condition (H) and is a regular pair. Then, the  following statements hold.
 \begin{eqnarray*}
 &&(i) \|P_nx-x\|\to 0, \ as \ n\to \infty\ for\ all \ x\in X;\\
 &&(ii) There \ exists \ a \ constant \ C>0\ such\ that \ \|P_n\|\leq C\ for \ all\ n=1,2,\cdots;\\
 &&(iii) \|P_nx-x\|\leq C\|Q_nx-x\| \ for\ some\ constant\ C>0 \ independent \ of \ n,
 \end{eqnarray*}
 where $Q_nx$ is the best approximation from $X_n$ to $x$, that is,
 \begin{equation*}
 \|x-Q_nx\|=\inf\limits_{x_n\in X_n} \|x-x_n\|.
 \end{equation*}
 \end{lemma}

With these lemmas in our hands, we are ready to explore the convergence rate of our Petrov-Galerkin method.
\begin{theorem}\label{th:eqPG1}
Let $u$ be the solution of \eqref{eqPG1} and $u_N$ be its corresponding Petrov-Galerkin approximation. If
$g_f(x):=f(x)e^{\sigma x}x^{1-\lambda-\nu}\in B^m_{w^{\lambda+\nu-1}}([0,\infty))$, then
\begin{equation}
\|D_s^{1-\nu}(u-u_N)\|\leq CN^{-m/2}\|\partial_x^mg_f\|_{w^{\lambda+\nu+m-1}}.
 \end{equation}
\end{theorem}
\begin{proof}
For any $\psi=x^{\lambda+\nu-1}e^{-\sigma x}v_N(x)\in V_1$, we have
\begin{equation}
(f-D_s^{1-\nu}u_N,\psi)_{\hat{w}_1}=0.
\end{equation}
We therefore deduce that $D_s^{1-\nu}u_N=\hat{\pi}_Nf$,
which  implies our problem is equivalent to find the best
approximation for $f$ in
\begin{equation}
X_n:=span\{x_n\}=span\{D_s^{1-\nu} \hat{L}_n^{\lambda} \}
=span\bigg\{ \frac{\Gamma(\lambda+n+1)}{\Gamma(\nu+\lambda+n)}\hat{L}_n^{\lambda+\nu-1} \bigg\}
\end{equation}
  by testing on $Y_n=V_1$.
   Define
\begin{equation}
\Pi_n x_n=y_n.
\end{equation}
A simple calculation shows that Lemma \ref{lem1} holds for our $X_n$ and $Y_n$ and furthermore, $\{X_n, Y_n\}$ is a regular pair with $C_1, C_2=\mathcal{O}(1)$.

From the definition of $\hat{L}_n$, we specify the operator
\begin{equation}
(\hat{\pi}_N f)(x):=x^{\lambda+\nu-1}e^{-\sigma x}(\pi_Ng_f)(x).
\end{equation}
Therefore,
\begin{equation}
0 = (\hat{\pi}_Nf-f, \psi)_{\hat{w}_1} = (\pi_Ng_f-g_f, v_N)_{w^{\lambda+\nu-1}}, \quad \forall \psi\in P_N.
\end{equation}
  This means our method essentially turns out to be a Galerkin method and it is equivalent to find a polynomial projection for the function
  $f(x)e^{\sigma x}x^{1-\lambda-\nu}$ with respect to the weight $x^{\lambda+\nu-1}e^{-2\sigma x}$.

Therefore, by Lemma \ref{Newlem1},
\begin{equation}
\|D_s^{1-\nu}(u-u_N)\|^2_{\hat{w}_1}=\|f-\hat{\pi}_Nf\|^2_{\hat{w}_1}
=\|\pi_Ng_f-g_f\|^2_{w^{\lambda+\nu-1}} \leq CN^{-m}\|\partial_x^mg\|_{w^{\lambda+\nu+m-1}}.
\end{equation}
\end{proof}

Similarly, for the fractional diffusion case, we have
\begin{theorem}
Let $u$ be  the solution of \eqref{eqPG2} and $u_N$ be its corresponding Petrov-Galerkin approximation. If
$g_f(x):=f(x)e^{\sigma x}x^{2-\lambda-\nu}\in B^m_{w^{\lambda+\nu-2}}([0,\infty)$, then
\begin{equation}
\|D_s^{2-\nu}(u-u_N)\|\leq CN^{-m/2}\|\partial_x^mg_f|_{w^{\lambda+\nu+m-2}}.
 \end{equation}
\end{theorem}
\begin{proof} The boundary condition has been satisfied automatically by our method. The analysis is the same as
that for the previous theorem and thus omitted. \end{proof}
\subsection{Numerical experiments}
\begin{example}
We first consider the equation \eqref{eqPG1} with $f(x)=B(7.3, \nu)/\Gamma(\nu)(\nu+6.3)x^{5.3+\nu}e^{-\sigma x}$ and $\sigma=2$.
Through the error estimate in Theorem 1, we conclude that the convergence rate of our Petrov-Galerkin method depends on
the regularity of the function $g(x)=f(x)e^{\sigma x}x^{1-\lambda-\nu}$ instead of $f(x)$ itself.  Therefore, we are allowed to adjust the parameter $\lambda$ according to the given data $f$ to maximize the smoothness of $g$.
In this case, $\lambda=0.3$ leads to an entire function $g(x)$, which further indicates an enhanced convergence rate. Indeed, for $\lambda=0.3$, the true solution sets root in our approximation space $U_1$ after a small $N$, and therefore only round-off errors are left. However, for other choice of $\lambda$\rq{}s, algebraic convergence rate is observed as the theorem predicts, see Figure 1. We also observe a convergence rate of $\mathcal{O}(N^{-7})$, which is better than our theoretical prediction  $\mathcal{O}(N^{-5.5})$ since $g(x)\in B^{11}_{w^{\lambda+\nu-1}}([0,\infty))$ for all $0<\lambda,\nu<1$.  In the experiment, the right hand side $F$ of our algorithm is approximated by $2N$-point Gauss-Laguerre numerical quadrature for each different $N$.
\begin{figure}[hp]
\thispagestyle{empty}
\centering
\resizebox{160mm}{75mm}{\includegraphics{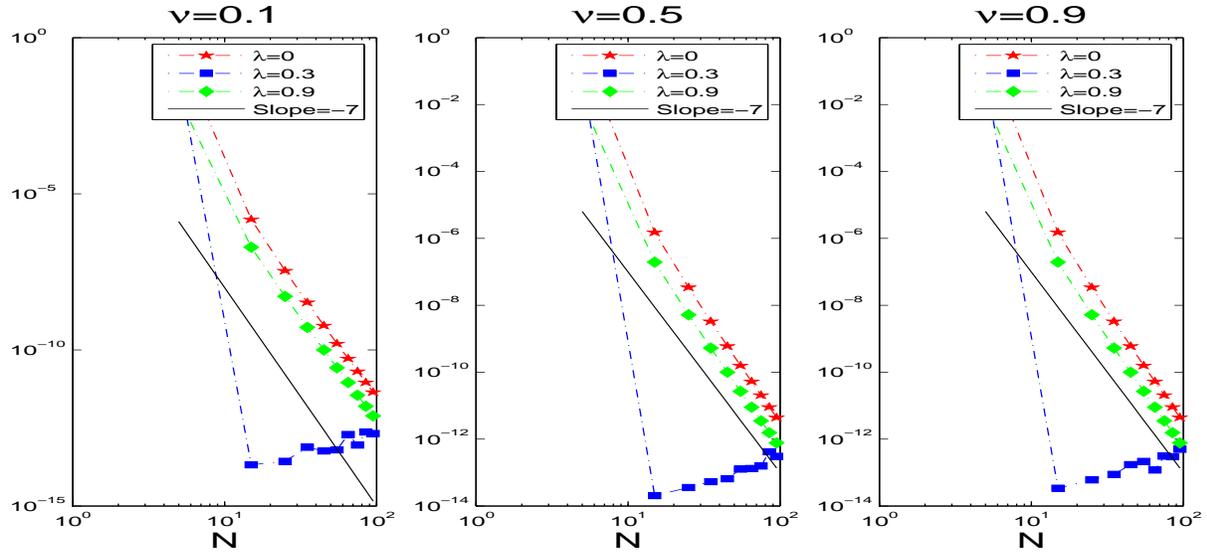}}
\caption{(Example 1): $L^\infty$ norm of numerical errors  to $D_{s}^{1-\nu} u(x)=f(x),\ x\in [0,\infty), \lim\limits_{x\to\infty}u(x)=0$,  for spectral Petrov-Galerkin method. }
\end{figure}
\end{example}

\begin{example}
Next, we consider equation \eqref{eqPG2} with true solution $f(x)=B(5.1, \nu)/\Gamma(\nu)(\nu+4.1)(\nu+3.1)x^{2.1+\nu}e^{-\sigma x}$ and $\sigma=2$.  Theorem 2 indicates that the convergence rate depends on the regularity of $g(x)=f(x)e^{\sigma x}x^{2-\lambda-\nu}\in B^{6}_{w^{\lambda+\nu-2}}([0,\infty))$ for all $0<\lambda,\nu<1$.  As expected, we only observe round-off errors for $\lambda=0.1$ after a small $N$. For  for other choice of $\lambda$\rq{}s, we observe algebraic convergence rate $\mathcal{O}(N^{-4})$, which is also better than our theoretical prediction $\mathcal{O}(N^{-3})$, see Figure 2 for details.   
\begin{figure}[tp]
\thispagestyle{empty}
\centering
\resizebox{160mm}{75mm}{\includegraphics{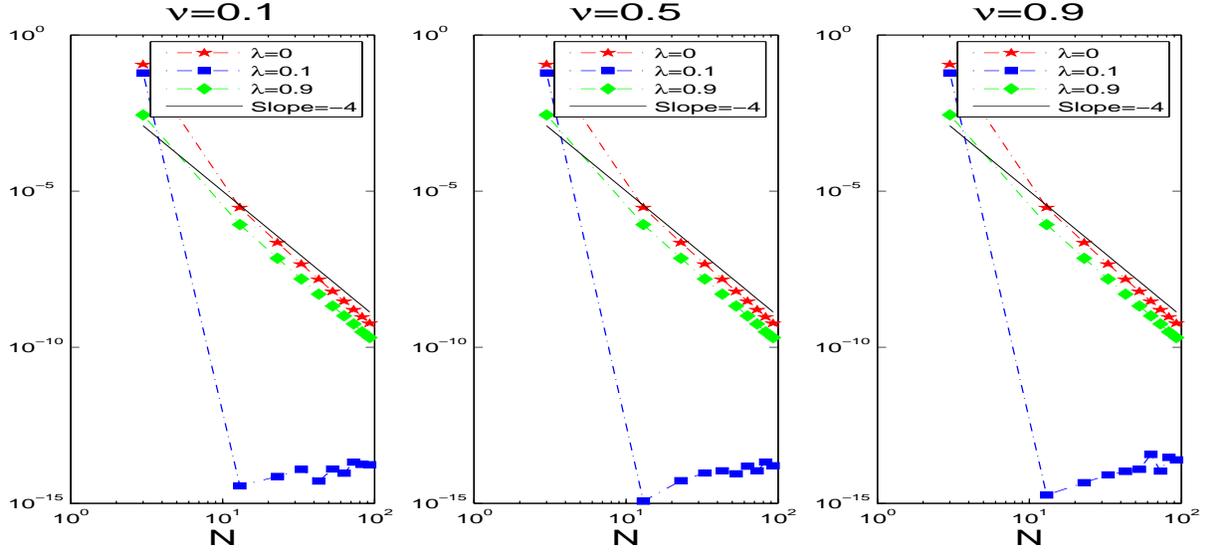}}
\caption{(Example 2): $L^\infty$ norm of numerical errors to $D_s^{2-\nu} u(x)=f(x),\ x\in [0,\infty), u(0)=0, \lim\limits_{x\to\infty}u(x)=0$ for spectral Petrov-Galerkin method. }
\end{figure}
\end{example}

\section{Substantial collocation matrix}
In this section, we elaborate on the construction of substantial collocation matrices regarding to \eqref{eqPG1} and \eqref{eqPG2}.
\subsection{Substantial fractional advection equation}
In our collocation method, we seek approximation of $u$
\begin{equation}
 u_N \in \text{span}\ \{\hat{L}_n^\lambda(x), 0\leq n\leq N\}
  \end{equation}
$x\in [0,\infty)$ of the form
\begin{equation}
u_N(x)=\sum\limits_{k=0}^N c_k \hat{L}_k^\lambda(x).
\end{equation}
For the sake of derivation of collocation matrix, we rewrite $u_N$ in nodal expansion.
 \begin{equation}
  u_N(x)=\sum\limits_{j=0}^N u(x_j) h_j(x),
 \end{equation}
 where $h_j(x)$ is our interpolant function defined as
 \begin{equation}
 \displaystyle{h_j(x)=\frac{x^\lambda e^{-\sigma x}}{x_j^\lambda e^{-\sigma x_j}}\prod\limits_{i=0, i\neq j}^N \frac{(x-x_i)}{(x_j-x_i)}}:=\frac{x^\lambda e^{-\sigma x}}{x_j^\lambda e^{-\sigma x_j}} l_j(x), \ 0\leq j\leq N.
\end{equation}
The associated points $\{x_j\}_{j=0}^N$ is  the Laguerre-Gauss points with respect to the weight $x^{\lambda+\nu-1}e^{-2\sigma x}$.

 Thereby, from \eqref{essentialI1} and the initial condition
\begin{eqnarray}
 D_s^{1-\nu} u_N(x)&=&\sum\limits_{j=1}^N u(x_j) D_s^{1-\nu} h_j(x)\\
                       &=&\sum\limits_{j=1}^N u(x_j) \frac{1}{x_j^\lambda e^{-\sigma x_j}} \sum\limits_{k=0}^N \beta_k^j [D_s^{1-\nu}\hat{L}_k^\lambda(x)]\nonumber\\
 &=& \sum\limits_{j=1}^N u(x_j) \frac{1}{x_j^\lambda e^{-\sigma x_j}} \sum\limits_{k=0}^N \beta_k^j  \frac{\Gamma(\lambda+k+1)}{\Gamma(\nu+\lambda+k)}\hat{L}_k^{\lambda+\nu-1}(x). \nonumber
\end{eqnarray}
Consequently, we evaluate the $D_s^{1-\nu} u_N(x)$ at collocation points and obtain
\begin{eqnarray}
D_s^{1-\nu} u_N(x)\bigg|_{x=x_i}&=& \sum\limits_{j=1}^N u(x_j) \frac{1}{x_j^\lambda e^{-\sigma x_j}} \sum\limits_{k=0}^N \beta_k^j  \frac{\Gamma(\lambda+k+1)}{\Gamma(\nu+\lambda+k)}\hat{L}_k^{\lambda+\nu-1}(x_i)\nonumber\\
                               &=&\sum\limits_{j=1}^N {\bf D}_{ij}  u(x_j),
 \end{eqnarray}
 where ${\bf D}_{ij}$ are the entries of the $N\times N$ collocation matrix $D$.

Next, let us find an explicit expression for $\beta_k^j$ such that $l_j(x)=\sum\limits_{k=0}^N \beta_k^j L_k^\lambda(2\sigma x)$.  It is clear that
\begin{equation}
l_j(x)=\sum\limits_{k=0}^N\alpha_k^jL_k^{\lambda+\nu-1}(2\sigma x),
\end{equation}
where
\begin{eqnarray}
\alpha_k^j&=&\frac{1}{\|L_k^{\lambda+\nu-1}(2\sigma x)\|_{w^{\lambda+\nu-1}}^2}\int_0^\infty l_j(x)L_k^{\lambda+\nu-1}(2\sigma x)x^{\lambda+\nu-1}e^{-2\sigma x}dx\nonumber\\
&=& \frac{(2\sigma)^{\lambda+\nu}\Gamma(k+1)}{\Gamma(k+\lambda+\nu)} w_jL_k^{\lambda+\nu-1}(2\sigma x_j)
\end{eqnarray}
since $(N+1)$-point Laguerre-Gauss quadrature is exact for all polynomials up to order $2N$. Furthermore, we denote
\begin{eqnarray}
L_k^{\lambda+\nu-1}(2\sigma x)&=&\sum\limits_{i=0}^k C_i^kL_i^\lambda(2\sigma x).
\end{eqnarray}
By a similar fashion,
\begin{eqnarray}
C_i^k&=&\frac{1}{\|L_i^{\lambda}(2\sigma x)\|_{w^{\lambda}}^2}\int_0^\infty L_i^\lambda(2\sigma x)L_k^{\lambda+\nu-1}(2\sigma x)x^{\lambda}e^{-2\sigma x}dx\nonumber\\
&=& \frac{\Gamma(i+1)}{\Gamma(i+\lambda+1)}\sum\limits_{n=0}^N w_n^\lambda L_k^{\lambda+\nu-1}(x_n^\lambda)L_i^\lambda(x_n^\lambda),
\end{eqnarray}
where $\{x_n^\lambda, w_n^\lambda\}_{n=0}^N$ are Laguerre-Gauss quadrature points and weights with respect to weight $x^\lambda e^{-x}$. We thereby obtain
\begin{eqnarray} \label{beta}
\beta_k^j&=& \sum\limits_{i=k}^NC_k^i\alpha_i^j.
  \end{eqnarray}

Hence, we obtain a closed form of the collocation matrix
\begin{equation} \label{advmatrix}
 {\bf D}_{ij}=\frac{w_j}{x_j^\lambda e^{-\sigma x_j}} \sum\limits_{k=0}^N  \beta_k^j\frac{\Gamma(\lambda+k+1)}{\Gamma(\nu+\lambda+k)}\hat{L}^{\lambda+\nu-1}_k(x_i).
 \end{equation}

\subsection{Substantial fractional diffusion equation}

In this subsection, we consider the spectral collocation method for \eqref{eqPG2}.  Like the previous subsection, we approximate $u$ by
\begin{equation}
u_N(x)=\sum\limits_{k=0}^{N-1} c_k \hat{L}_k^{\lambda}(x)
\end{equation}
such that it satisfies the initial condition. Let $\{x_i,w_i\}_{i=1}^N$ be the $(N+1)$-point Laguerre-Gauss-Radau points with respect to the weight $x^{\lambda}e^{-2\sigma x}$ with $0$ excluded. As before, we rewrite it in nodal expansion form
\begin{equation}
u_N(x)=\sum\limits_{j=1}^N u(x_j)h_j(x),
\end{equation}
where $h_j(x)$ is of the form
\begin{equation}
h_j(x)=\frac{x^{\lambda}e^{-\sigma x}}{x_j^{\lambda}e^{-\sigma x_j}}\prod\limits_{i=1,i\neq j}^N \bigg(\frac{x-x_i}{x_j-x_i}\bigg):=\frac{x^{\lambda}e^{-\sigma x}}{x_j^{\lambda}e^{-\sigma x_j}}l_j(x), 1\leq j \leq N.
\end{equation}
Hence, by \eqref{essentialI2},
\begin{eqnarray}
D_s^{2-\nu}u_N(x)&=& \sum\limits_{j=1}^N u(x_j) D_s^{2-\nu}h_j(x)\nonumber\\
&=& \sum\limits_{j=1}^N u(x_j) \frac{1}{x_j^{\lambda}e^{-2\sigma x_j}}\sum\limits_{k=0}^{N-1} \beta_k^j D_s^{2-\nu}[\hat{L}_k^{\lambda}(x)]\nonumber\\
&=& \sum\limits_{j=1}^N u(x_j) \frac{1}{x_j^{\lambda}e^{-2\sigma x_j}}\sum\limits_{k=0}^{N-1} \beta_k^j \frac{\Gamma(\lambda+k+1)}{\Gamma(\nu+\lambda+k-1)}\hat{L}_k^{\lambda+\nu-2}(x).
\end{eqnarray}
Since $0$ is excluded in the collocation points set, the way to find $\beta_k^j$ is different from \eqref{beta}.
\begin{eqnarray} \label{beta1}
\beta_k^j&=& \frac{(2\sigma)^{1+\lambda}k!}{\Gamma(k+\lambda+1)}\int_0^\infty l_j(x)L_k^{\lambda}(x)x^{\lambda}e^{-2\sigma x}dx, \nonumber\\
&=& \frac{(2\sigma)^{1+\lambda}k!}{\Gamma(k+\lambda+1)} \bigg[w_0l_j(0)\frac{\Gamma(k+\lambda+1)}{k!\Gamma(\lambda+1)}+w_jL_k^{\lambda}(2\sigma x_j)\bigg].                                                                                                                                                                                                                                                                                                                             \end{eqnarray}
From the orthogonality of Laguerre polynomials and the fact that the Laguerre-Gauss-Radau is exact for all polynomial of order up to $2N$.
\begin{eqnarray}\label{eq:lj0}
0&=& \int_0^\infty l_j(x)L_N^{\lambda}(2\sigma x)x^{\lambda}e^{-2\sigma x}dx,\nonumber\\
&=& w_0l_j(0)\frac{\Gamma(N+\lambda+1)}{N!\Gamma(\lambda+1)}+w_jL_N^{\lambda}(2\sigma x_j).
\end{eqnarray}
Solve for $l_j(0)$ from \eqref{eq:lj0} and substitute it into  \eqref{beta1},
\begin{equation}\label{beta1value}
\beta_k^j=(2\sigma)^{\lambda+1}w_j\bigg[\frac{k!L_k^{\lambda}(2\sigma x_j)}{\Gamma(k+\lambda+1)}-\frac{N!L_N^{\lambda}(2\sigma x_j)}{\Gamma(N+\lambda+1)}\bigg].
\end{equation}
We then obtain the collocation matrix
\begin{eqnarray}\label{diffusmatrix}
D_{ij}&=& \frac{w_j(2\sigma)^{\lambda+1}}{x_j^{\lambda}e^{-\sigma x_j}}\sum\limits_{k=0}^{N-1}
\bigg[\frac{k!L_k^{\lambda}(2\sigma x_j)}{\Gamma(k+\lambda+1)}-\frac{N!L_N^{\lambda}(2\sigma x_j)}{\Gamma(N+\lambda+1)}\bigg] \frac{\Gamma(\lambda+k+1)}{\Gamma(\nu+\lambda+k-1)}\hat{L}_k^{\lambda+\nu-2}(x_i)
\end{eqnarray}
where $1\leq i,j \leq N.$

\begin{remark}
 The computational cost of obtaining the $D$ matrix from (\ref{advmatrix}) or (\ref{diffusmatrix}) may be high since it is a full matrix in general. However, this task can be done once for all. The strength is remarkable if one solves a system of fractional differential equations repeatedly.
\end{remark}

\subsection{Convergence analysis}
In this subsection, we develop convergence analysis for collocation method of advection equation and diffusion equation separately because the former is essentially a Galerkin method and is associated with the analysis in section 3 whereas the analysis of the latter is relatively new.

\subsubsection{Fractional advection equation}
Before start, we introduce an estimate on the interpolation error on Gauss-Laguerre points.
\begin{lemma}\cite[Page 272]{STW11}\label{intererror}
Let $\alpha>-1$. If $u\in C([0,\infty))\cap B_{w^\alpha}^m([0,\infty))$ and $\partial_xu\in B_{w^\alpha}^{m-1}([0,\infty))$ with $1\leq m\leq N+1$, then
\begin{equation}
\|I_N^\alpha u-u\|_{w^\alpha}\leq C\sqrt{\frac{(N-m+1)!}{N!}}(\|\partial_x^m u\|_{w^{\alpha+m-1}}+(\ln N)^{1/2}\|\partial_x^m u\|_{w^{\alpha+m}}),
\end{equation}
where $I_N^\alpha$ is the interpolation operator on $(N+1)$-Gauss-Laguerre points with respect to $x^\alpha e^{-x}$.
\end{lemma}

\begin{theorem}
 Let $u$ and $u_N$ be the solution of \eqref{eqPG1} and its collocation method with the $D$ matrix given by \eqref{advmatrix}, respectively. Let $g_f(x) = f(x)e^{\sigma x}x^{1-\lambda-\nu} \in B^m_{w^{\lambda+\nu-1}}([0,\infty))$.
 Then
\begin{equation}
\|D_s^{1-\nu}(u-u_N)\|_{\hat{w}^{1-\lambda-\nu}}\leq CN^{-m/2}(\|\partial_x^mg_f\|_{w^{\lambda+\nu+m-2}}
+ (\ln N)^{1/2}\|\partial_x^mg_f\|_{w^{\lambda+\nu+m-1}}).
\end{equation}
\end{theorem}
\begin{proof}
Recall that we collocate the equation on $(N+1)$ Laguerre-Gauss points associated with weight $x^{\lambda+\nu-1}e^{-2\sigma x}$.
 \begin{equation}
D_s^{1-\nu}u_N(x_i)=f(x_i), \quad i=0,\cdots, N.
\end{equation}
Multiplying both sides of the above equation by $x_i^{1-\lambda-\nu}{L}_n^{\lambda+\nu-1}(x_i)w_i$ and summing up, we obtain
\begin{equation}
\sum\limits_{i=0}^N D_s^{1-\nu}u_N(x_i)x_i^{1-\lambda-\nu}e^{\sigma x_i}{L}_n^{\lambda+\nu-1}(2\sigma x_i)w_i=\sum\limits_{i=0}^N f(x_i)x_i^{1-\lambda-\nu}e^{\sigma x_i}{L}_n^{\lambda+\nu-1}(2\sigma x_i)w_i,
\end{equation}
which is equivalent to
\begin{equation}
( D_s^{1-\nu}u_N, \hat{L}_n^{\lambda+\nu-1})_{\hat{w}^{1-\lambda-\nu}}
=(g_f,{L}_n^{\lambda+\nu-1}(2\sigma x))_{w^{\lambda+\nu-1}}^*
=(I_Ng_f, L_n^{\lambda+\nu-1}(2\sigma x))_{w^{\lambda+\nu-1}}
\end{equation}
 from the exactness of the $(N+1)$-point Laguerre-Gauss quadrature. Here, $*$ indicates numerical quadrature.  Then, the Strang\rq{}s first lemma  leads to
\begin{eqnarray}
&&\|f-D_s^{1-\nu}u_N\|_{\hat{w}^{1-\lambda-\nu}}\leq C\bigg(\inf\limits_{f_N\in V_1}\|f-f_N\|_{\hat{w}^{1-\lambda-\nu}}+\sup\limits_{w_N\in V_1}\frac{|(f,w_N)_{\hat{w}^{1-\lambda-\nu}}-(f,w_N)^*|}{\|w_N\|_{\hat{w}^{1-\lambda-\nu}}}\bigg)\\
&\leq& C\bigg(\inf\limits_{f_N\in V_1}\|f-f_N\|_{\hat{w}^{1-\lambda-\nu}}+\sup\limits_{q_N\in P_N}\frac{|(g_f,q_N)_{w^{\lambda+\nu-1}}-(I_Ng_f,q_N)_{w^{\lambda+\nu-1}}|}{\|q_N\|_{w^{\lambda+\nu-1}}}\bigg)\notag\\
&\leq& C(\|f-\hat{\pi}_Nf\|_{\hat{w}^{1-\lambda-\nu}}+\|g_f-I_Ng_f\|_{w^{\lambda+\nu-1}}).
\end{eqnarray}
Then the result follows by Theorem \ref{th:eqPG1} and Lemma \ref{intererror}.
\end{proof}
\subsubsection{Fractional diffusion equation}
Unlike the previous subsection,  we collocate \eqref{eqPG2} on Laguerre-Gauss-Radau points $\{x_i\}_{i=0}^N$ with respect to the weight $x^\lambda e^{-2\sigma x}$ instead of $x^{\lambda+\nu-2}e^{-2\sigma x}$ because if $\lambda+\nu-2\leq -1$, those quadrature points and weights may not exist. Thereby our collocation method is an authentic Petrov-Galerkin method.

\begin{theorem}
 Let $u$ and $u_N$ be the solution of \eqref{eqPG2} and its collocation method with $D$ matrix given by \eqref{diffusmatrix}, respectively. Let $g_f(x) = f(x)e^{\sigma x}x^{2-\lambda-\nu} \in B^m_{w^{\lambda+\nu-2}}([0,\infty))$. Then
\begin{equation}
\|D_s^{2-\nu}(u-u_N)\|_{\hat{w}^{2-\lambda-\nu}} \leq CN^{-m/2}(\|\partial_x^mg_f\|_{\hat{w}^{\lambda+m-1}}
+ (\ln N)^{1/2}\|\partial_x^mg_f\|_{\hat{w}^{\lambda+m}}).
\end{equation}
\end{theorem}
\begin{proof}
As before, we take the modal expansion of $u$
\begin{equation}
u_N(x):=\sum\limits_{j=0}^N u_N(x_j)h_j(x)=\sum\limits_{k=0}^N c_k\hat{L}_k^{\lambda}(2\sigma x)
 \end{equation}
and perform the collocation,
 \begin{equation}
D_s^{2-\nu}u_N(x_i)=f(x_i), \quad i=0,\cdots, N.
\end{equation}
Multiplying both sides of the above equation by $x_i^{2-\lambda-\nu}e^{\sigma x_i}{L}_n^{\lambda}(2\sigma x_i)w_i$ and summing up, we obtain
\begin{equation}
\sum\limits_{i=0}^N D_s^{2-\nu}u_N(x_i)x_i^{2-\lambda-\nu}e^{\sigma x_i}{L}_n^{\lambda}(2\sigma x_i)w_i=\sum\limits_{i=0}^N f(x_i)x_i^{2-\lambda-\nu}e^{\sigma x_i}{L}_n^{\lambda}(2\sigma x_i)w_i,
\end{equation}
which is equivalent to
\begin{equation}
(D_s^{2-\nu}u_N, \hat{L}_n^{\lambda})_{\hat{w}^{2-\lambda-\nu}}
=(g_f, {L}_n^{\lambda}(2\sigma x))_{w^{\lambda}}^*=(I_N g_f, L_n^{\lambda}(2\sigma x))_{w^{\lambda}}
\end{equation}

To proceed, we consider an auxiliary equation
\begin{equation}
(D_s^{2-\nu}\bar{u}_N, \hat{L}_n^{\lambda})_{\hat{w}^{2-\lambda-\nu}}=(f, \hat{L}_n^{\lambda})_{\hat{w}^{2-\lambda-\nu}}.
\end{equation}
Denote $ (D_s^{2-\nu}\bar{u}_N)(x):=(\hat{\pi}_Nf)(x)=x^{\lambda+\nu-2}e^{-2\sigma x}(\pi_Ng_f)(x)$. The problem is simplified to
\begin{equation}
(f-\hat{\pi}_Nf, \hat{L}_n^{\lambda})_{\hat{w}^{2-\lambda-\nu}}=0,
\end{equation}
which is clear a Petrov-Galerkin method with
\begin{eqnarray}
&&X_n=span\bigg\{\frac{\Gamma(n+\lambda+1)}{\Gamma(n-1+\lambda+\nu)}\hat{L}_n^{\lambda+\nu-2}(x)\bigg\}\nonumber\\ &&Y_n=\{\hat{L}_n^\lambda(x)\}
\end{eqnarray}
A careful calculation indicates that $\{X_n,Y_n\}$ is a regular pair, see Appendix B. By Lemma \ref{Newlem1} and Lemma \ref{lem2},
if $g_f\in B^m_{w^{2-\lambda-\nu}}([0,\infty))$,
\begin{equation}\label{eq:diffu1}
\|D_s^{2-\nu}u-D_s^{2-\nu}\bar{u}_N\|_{w^{2-\lambda-\nu}}=\|f-\hat{\pi}_Nf\|_{w^{2-\lambda-\nu}}\leq CN^{-m/2}\|\partial_x^mg_f\|_{w^{\lambda+\nu+m-1}}.
\end{equation}
Now, let us consider the effect of numerical integration. By Lemma \ref{intererror},
\begin{eqnarray}\label{eq:diffu2}
&&\sup\limits_{L_n^\lambda\in P_N}\frac{|(f, \hat{L}_n^{\lambda})_{\hat{w}^{2-\lambda-\nu}}
-(I_Ng_f, L_n^{\lambda}(2\sigma x))_{w^{\lambda}}|}{ \|L_n^\lambda\|_{w^\lambda}}
= \sup\limits_{L_n^\lambda\in P_N} \frac{|(g_f-I_Ng_f, L_n^{\lambda} (2\sigma x))_{w^{\lambda}}|} {\| L_n^\lambda \|_{w^\lambda}} \nonumber\\
&&\leq C\sqrt{\frac{(N-m+1)!}{N!}}(\|\partial_x^mg_f\|_{w^{\lambda+m-1}}+(\ln N)^{1/2}\|\partial_x^mg_f\|_{w^{\lambda+m}}),
\end{eqnarray}
Accounting of \eqref{eq:diffu1} and \eqref{eq:diffu2}, the result is then followed by the Strang\rq{}s lemma.
\end{proof}

\subsection{Numerical experiments}
\begin{example}
In order to test our collocation matrix $D$, we first consider the simplest substantial FDE \eqref{eqPG1}. Collocating the equation on the Laguerre-Gauss points associated with the weight $x^{\lambda+\nu-1}e^{-\sigma x}$, we obtain the collocation matrix based upon \eqref{advmatrix}.
In this example, we choose $\sigma=2$ and $f(x)=B(7.3, \nu)/\Gamma(\nu)(\nu+6.3)x^{\nu+5.3}e^{-\sigma x}$. Numerical behaviors for different number of collocation points are presented in Figure 3, where we observe that the condition number of the collocation matrix grows mildly with respect to the differential index $\nu$.
\end{example}
\afterpage{
\begin{figure}[hp]
\thispagestyle{empty}
\centering
\resizebox{110mm}{65mm}{\includegraphics{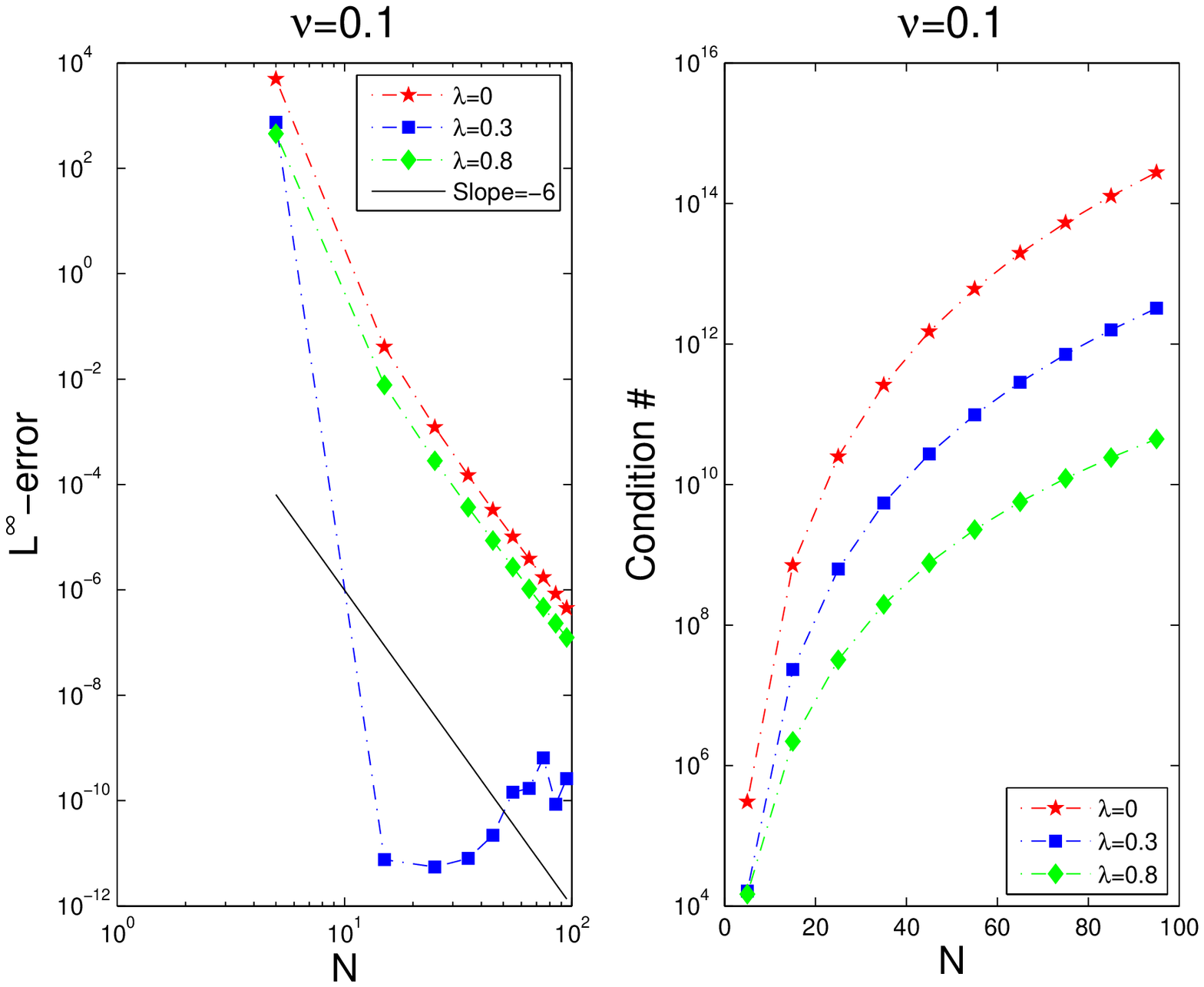}}
\resizebox{110mm}{65mm}{\includegraphics{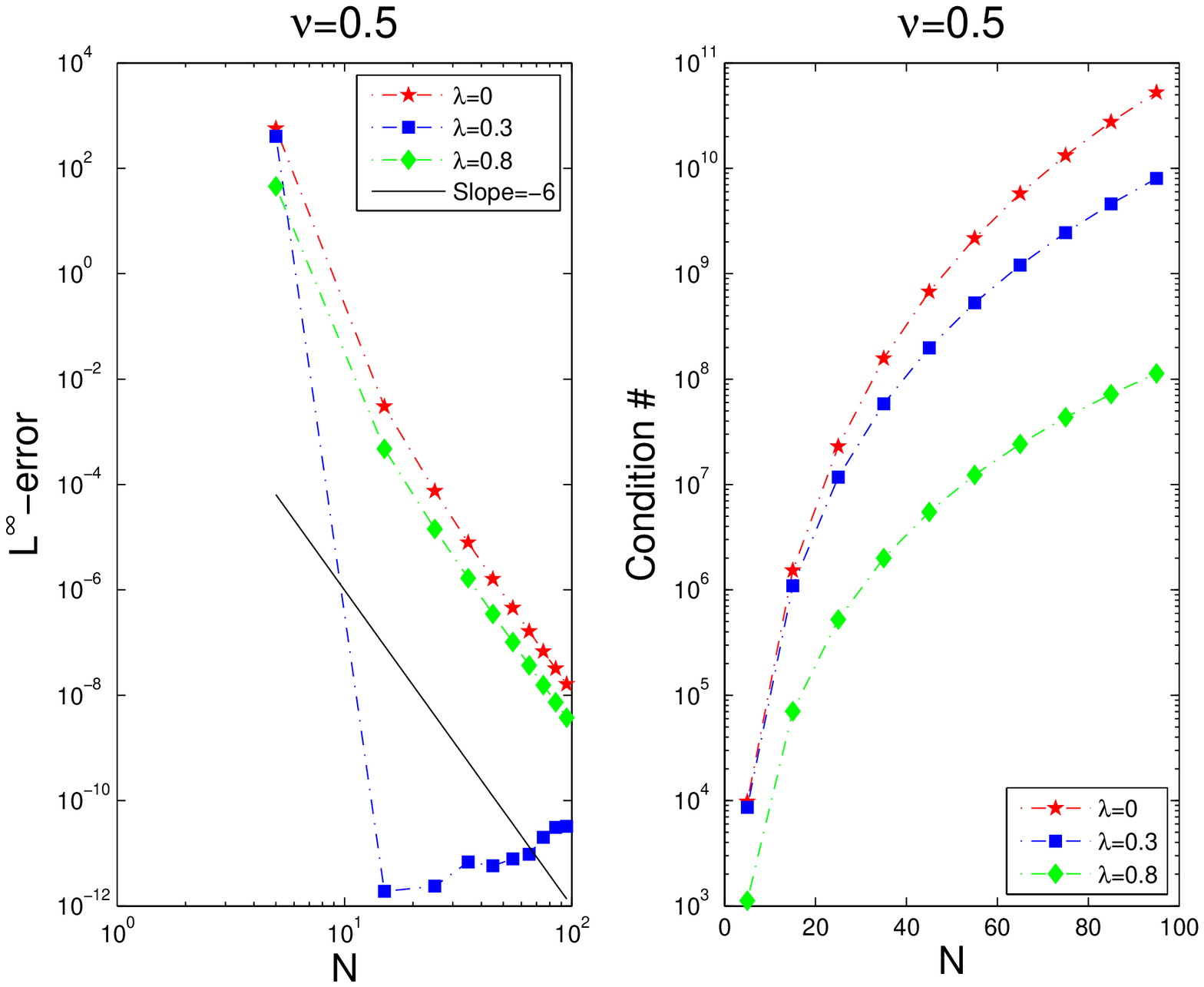}}
\resizebox{110mm}{65mm}{\includegraphics{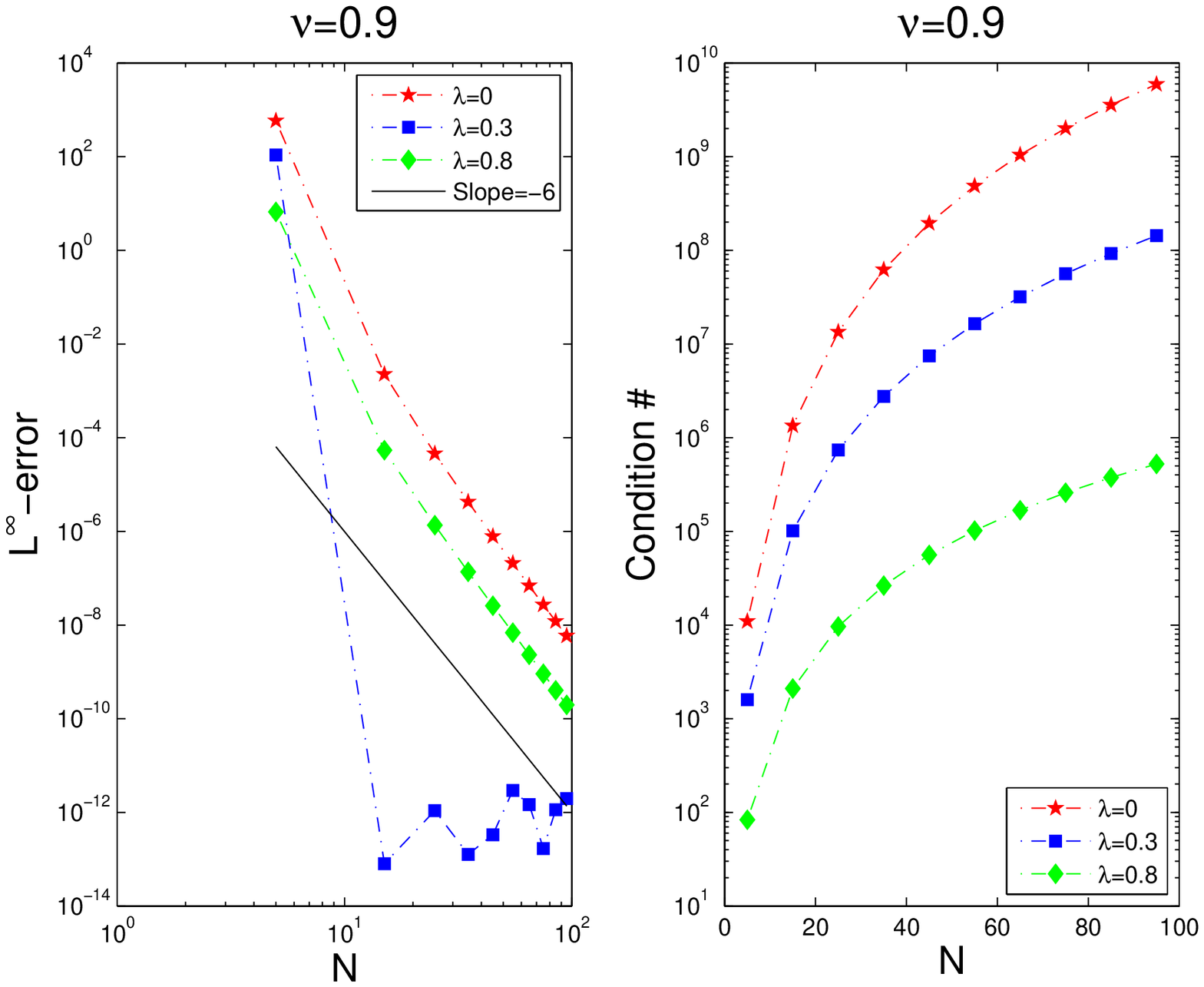}}
\caption{(Example 4): $L^\infty$ norm of numerical errors to $D_s^{2-\nu} u(x)=f(x),\ x\in [0,\infty), u(0)=0, \lim\limits_{x\to \infty}u(x)=0$ for Laguerre-Gauss-Radau points. The right column is plots of condition numbers of direct collocation differential matrix.}
\end{figure}
\clearpage}

\begin{example}
For \eqref{eqPG2}, we choose $f(x)=B(7.3,\nu)/\Gamma(\nu)(\nu+6.3)(\nu+5.3)x^{4.3+\nu}e^{-\sigma x}$ and collocate the equation on Laguerre-Gauss-Radau points with respect to the weight $x^\lambda e^{-2\sigma x}$ with $\sigma=2$.  As predicted by Theorem 4, we only observe round-off errors for $\lambda=0.3$ and  algebraic convergence rate for other $\lambda$\rq{}s, see Figure 4. However, we also observe that the condition number of resulting system grows dramatically as $N$ grows.

\begin{figure}[hp]
\thispagestyle{empty}
\centering
\resizebox{110mm}{65mm}{\includegraphics{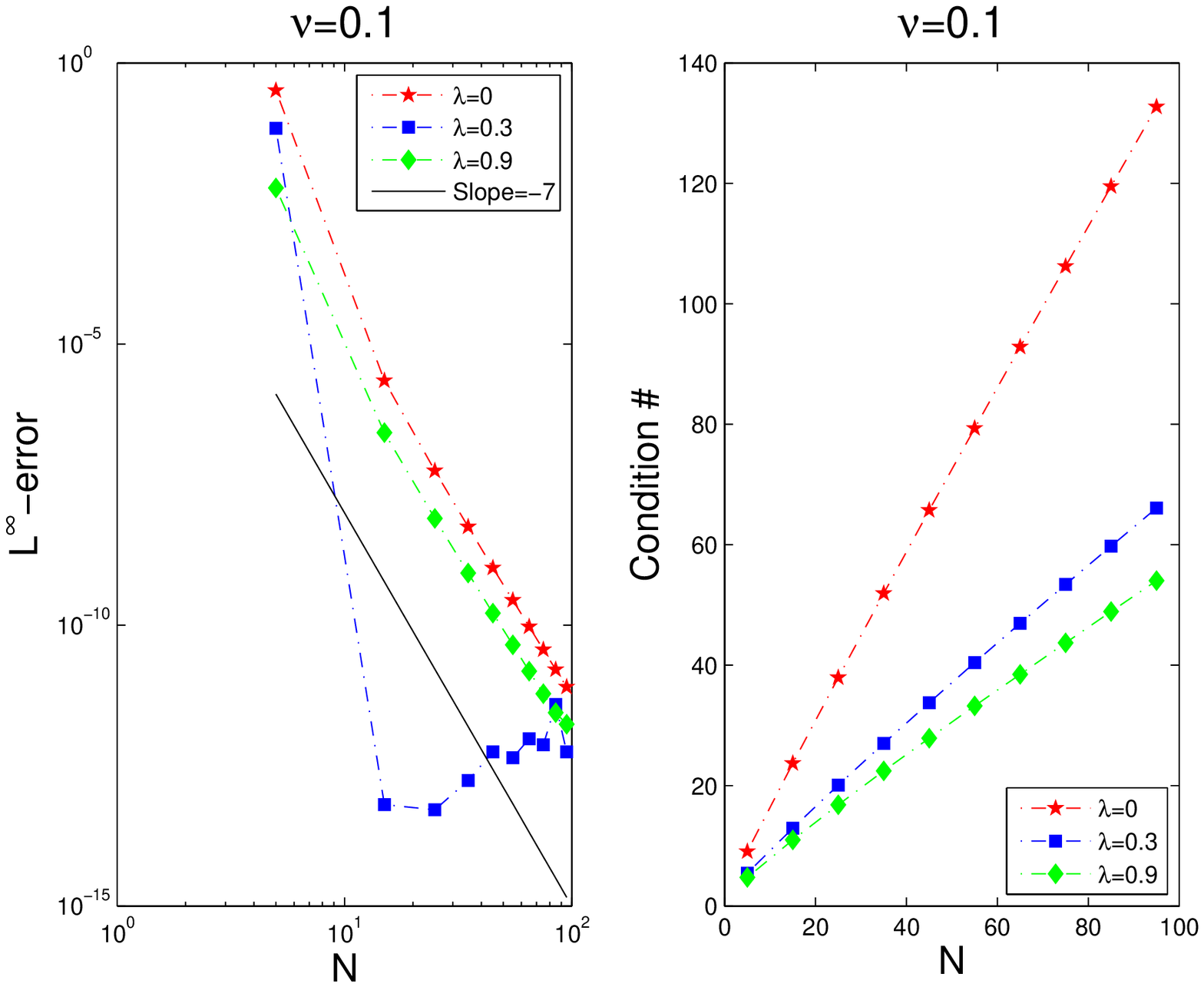}}
\resizebox{110mm}{65mm}{\includegraphics{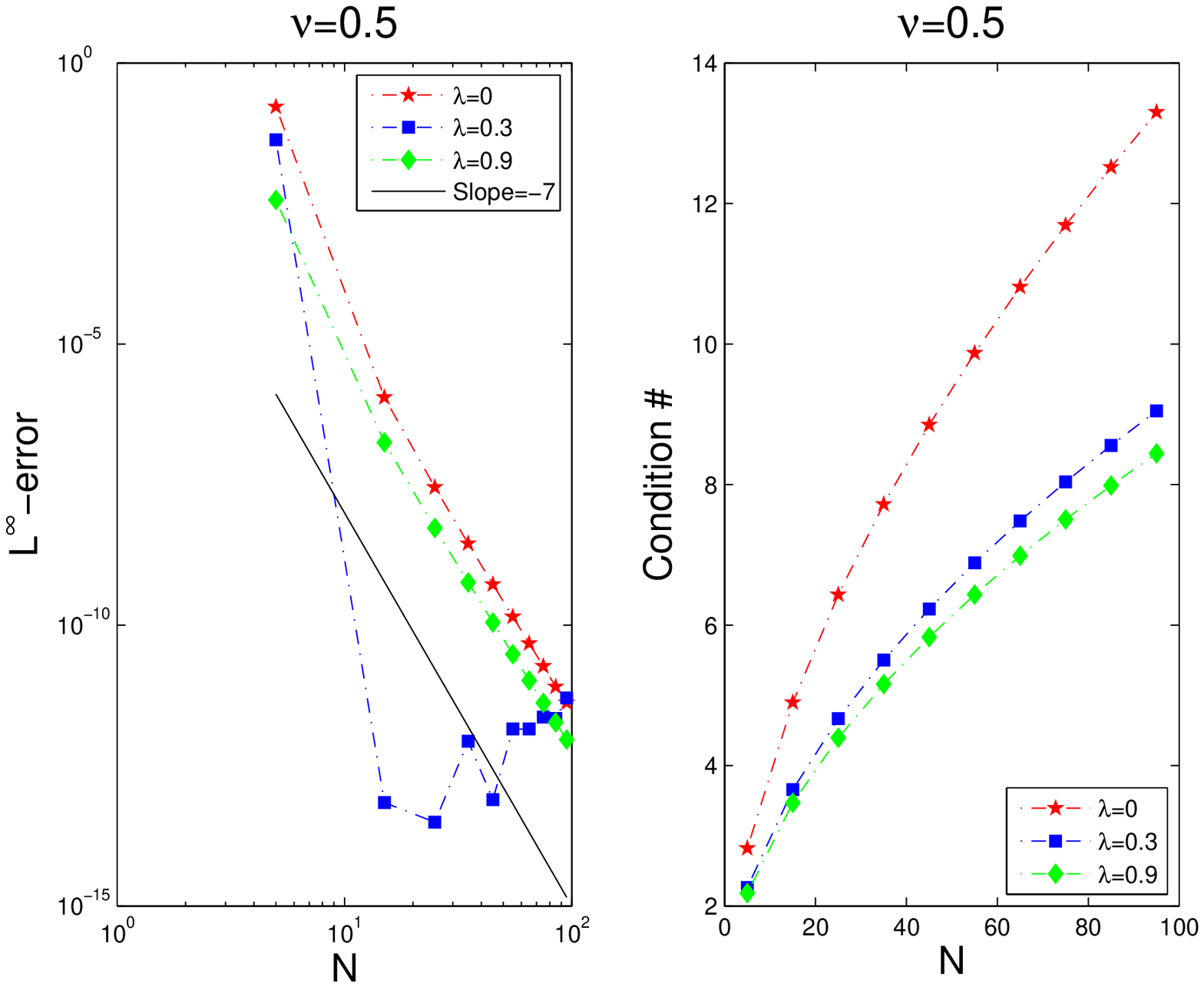}}
\resizebox{110mm}{65mm}{\includegraphics{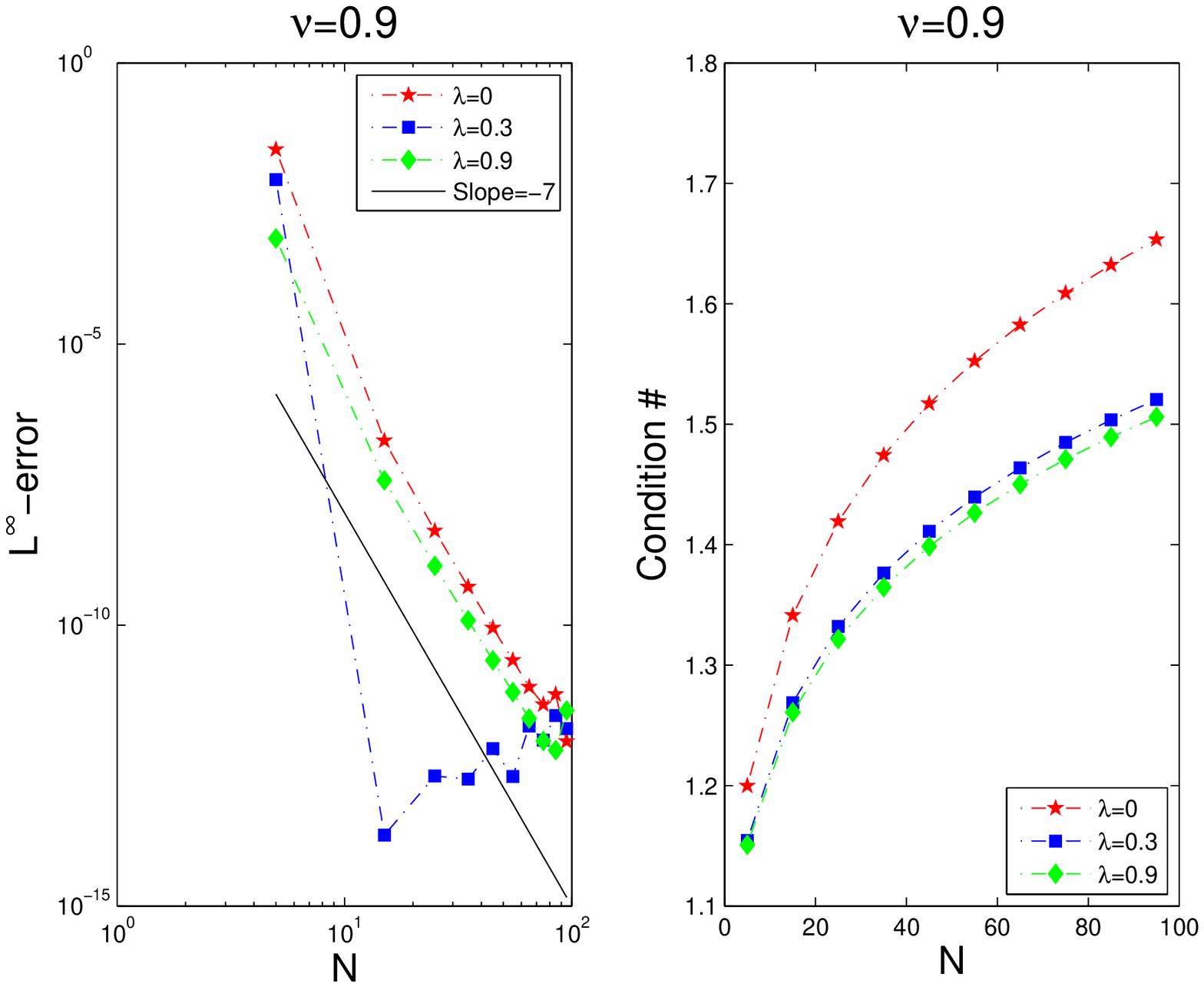}}
\caption{(Example 3): $L^\infty$ norm of numerical errors to $D_s^{1-\nu} u(x)=f(x),\ x\in [0,\infty), \lim\limits_{x\to\infty}u(x)=0$ for Laguerre-Gauss points. The right column is plots of condition numbers of direct collocation differential matrix.}
\end{figure}
\end{example}


\section{Conclusion}
We have considered both Petrov-Galerkin method and spectral collocation method for two types of substantial fractional differential equations. Four different algorithms for the model problems have been proposed, analyzed and tested. Our main  contributions are rooted in:

 $\bullet$ Extension of classical generalized Laguerre polynomials, which serves as a natural basis for our second order problem. Unlike the extension in \cite{G13}, our method inherits some basic properties of standard Laguerre polynomials, such as three-term recursion and orthogonality.

 $\bullet$ Construction of well-conditioned Petrov-Galerkin methods for fractional advection/diffusion equations, which yields diagonal linear systems. Error estimates are derived with convergence rate only depending on the smoothness of the given data.

$\bullet$ Construction of spectral collocation method for these two types of equations. Explicit collocation matrices are developed and associated error estimates are carried out.

The error estimates indicate that we are able to adjust the parameter $\lambda$ in the trial basis for a given data in our algorithms such that  an optimal convergence rate is obtained.
\appendix

\makeatother
\setcounter{equation}{0}
\renewcommand{\theequation}{A. \arabic{equation}}
\section{Properties of extended Laguerre polynomial}
We verify some important properties of our extension of Laguerre polynomials with $-2<\alpha\leq-1$ in this appendix.

\begin{equation}
\bullet\  \text{Recursive\ Relation}:\ (n+1)L_{n+1}^\alpha(x)=(2n+1+\alpha-x)L_n^\alpha(x)-(n+\alpha)L_{n-1}^\alpha(x). \quad\quad\quad\quad\quad\quad\quad
\end{equation}
\begin{proof}
For the case, $L_n^{\alpha+1}(x)$ is a standard polynomial and satisfies the standard three-term relation. Therefore,
\begin{eqnarray}
&&(2n+1+\alpha-x)L_n^\alpha(x)-(n+\alpha)L_{n-1}^\alpha(x)\\
&=&(2n+1+\alpha-x)(L_n^{\alpha+1}(x)-L_{n-1}^{\alpha+1}(x))-(n+\alpha)(L_{n-1}^{\alpha+1}(x)-L_{n-2}^{\alpha+1}(x))\nonumber\\
&=&\bigg[(2n+2+\alpha-x)L_n^{\alpha+1}(x)-(n+\alpha+1)L_{n-1}^{\alpha+1}(x)\bigg]-\bigg[(2n+\alpha-x)L_{n-1}^{\alpha+1}(x)-(n+\alpha)L_{n-2}^{\alpha+1}(x)\bigg]\nonumber\\
&&-L_n^{\alpha+1}(x)\nonumber\\
&=& (n+1)L_{n+1}^{\alpha+1}(x)-(n+1)L_n^{\alpha+1}(x)\nonumber\\
&=&(n+1)L_{n+1}^\alpha(x).\nonumber
\end{eqnarray}
\end{proof}
\begin{equation}
\bullet\  \text{The\  Sturn-Liouville\  equation}:\ -\partial_x(x^{\alpha+1}e^{-x}\partial_xL_n^\alpha(x))=nx^\alpha e^{-x}L_n^\alpha(x). \quad\quad\quad\quad\quad\quad
\end{equation}
\begin{proof}
We first check a useful identity by induction.
\begin{equation}\label{app:eq}
xL_{n-1}^{\alpha+1}(x)=(n+\alpha)L_{n-1}^\alpha(x)-nL_n^\alpha(x).
\end{equation}
By definition, the statement is equivalent to
\begin{equation}
(2n+\alpha)L_{n-1}^{\alpha+1}(x)-nL_n^{\alpha+1}(x)-(n+\alpha)L_{n-2}^{\alpha+1}(x)=xL_{n-1}^{\alpha+1}(x),
\end{equation}
which is obvious by an induction on $n$.  Hence, \eqref{app:eq} is valid for all $\alpha>-2$. Now, let us consider the Sturn-Liouville equation.
\begin{eqnarray}
 &&-\partial_x(x^{\alpha+1}e^{-x}\partial_xL_n^\alpha(x))\nonumber\\
 &=&\partial_x(x^{\alpha+1}e^{-x}(L_{n-1}^{\alpha+2}(x)-L_{n-2}^{\alpha+2}(x))\nonumber\\
 &=&\partial_x(x^{\alpha+1}e^{-x}L_{n-1}^{\alpha+1}(x))\nonumber\\
 &=&x^{\alpha}e^{-x}[(\alpha+1)L_{n-1}^{\alpha+1}(x)-xL_{n-1}^{\alpha+1}(x)-xL_{n-2}^{\alpha+2}(x)]\nonumber\\
 &=&x^{\alpha}e^{-x}[(\alpha+1)L_{n-1}^{\alpha+1}(x)-xL_{n-1}^{\alpha+1}(x)-(n+\alpha)L_{n-2}^{\alpha+1}(x)+(n-1)L_{n-1}^{\alpha+1}(x)]\nonumber\\
 &=&x^{\alpha}e^{-x}[(n+\alpha)(L_{n-1}^{\alpha+1}(x)-L_{n-2}^{\alpha+1}(x))-xL_{n-1}^{\alpha+1}(x)]\nonumber\\
 &=&x^{\alpha}e^{-x}[(n+\alpha)L_{n-1}^\alpha(x)-xL_{n-1}^{\alpha+1}(x)]\nonumber\\
 &=&x^{\alpha}e^{-x}[(n+\alpha)L_{n-1}^\alpha(x)-(n+\alpha)L_{n-1}^\alpha(x)+nL_n^\alpha(x)]\nonumber\\
 &=&n x^\alpha e^{-x}L_n^\alpha(x).
  \end{eqnarray}
 \end{proof}
\begin{equation}
\bullet\ \text{Orthogonality} \int_0^\infty \partial_x^m L_n^\alpha(x)\partial_x^m L_k^\alpha(x)x^{\alpha+m}e^{-x}dx=\frac{\Gamma(n+\alpha+1)}{\Gamma(n-m+1)}\delta_{n,k}, m\geq 0. \quad\quad\quad\quad
\end{equation}
\begin{proof}
We only need to verify the case for $m=0$ since others are followed by property of  standard Laguerre polynomials.
It is clear that $L_n^\alpha(x)$ satisfy the Strun-Liouville equation. Hence, by the Sturn-Liouville theory $L_n^\alpha(x)$ and $L_k^\alpha(x)$ are orthogonal to each other with respect to the weight $x^\alpha e^{-x}$ if $n\neq k$. For $n=k$,
we need to prove
\begin{equation}\label{app:eq0}
\int_0^\infty [L_n^\alpha(x)]^2x^{\alpha}e^{-x}dx=\frac{\Gamma(n+\alpha+1)}{\Gamma(n+1)}.  \quad\quad\quad\quad
\end{equation}
We first derive three useful identities.
\begin{eqnarray}
\int_0^\infty x^2[L_k^\alpha(x)]^2 x^\alpha e^{-x}dx&=&\int_0^\infty [L_k^{\alpha+1}(x)-L_{k-1}^{\alpha+1}(x)]^2 x^{\alpha+2}e^{-x}dx\nonumber\\
&=&\int_0^\infty [L_k^{\alpha+2}-2L_{k-1}^{\alpha+2}(x)+L_{k-2}^{\alpha+2}(x)]^2x^{\alpha+2}e^{-x}dx\nonumber\\
&=&\frac{\Gamma(k+\alpha+3)}{\Gamma(k+1)}+4\frac{\Gamma(k+\alpha+2)}{\Gamma(k)}+\frac{\Gamma(k+\alpha+1)}{\Gamma(k-1)}. \label{app:eq1}\\
\int_0^\infty [L_k^\alpha(x)]^2 x^{\alpha+1}e^{-x}dx&=&\int_0^\infty [L_k^{\alpha+1}(x)-L_{k-1}^{\alpha+1}(x)]^2 x^{\alpha+1}e^{-x}dx\nonumber\\
&=&\frac{\Gamma(k+\alpha+2)}{\Gamma(k+1)}+\frac{\Gamma(k+\alpha+1)}{\Gamma(k)}. \label{app:eq2}\\
\int_0^\infty L_k^\alpha(x)L_{k-1}^\alpha(x) x^{\alpha+1}e^{-x}dx&=&\int_0^\infty [L_k^{\alpha+1}(x)-L_{k-1}^{\alpha+1}(x)]
[L_{k-1}^{\alpha+1}(x)-L_{k-2}^{\alpha+1}(x)]\nonumber\\
&=&-\frac{\Gamma(k+\alpha+1)}{\Gamma(k)}.\label{app:eq3}
\end{eqnarray}
Again, we prove \eqref{app:eq0} by induction. The case $n=0$ follows  the definition the gamma function. Suppose when $n<k$, the identity holds.
Now, let us verify $n=k+1$,
\begin{eqnarray}
&&\int_0^\infty [L_{k+1}^\alpha(x)]^2x^\alpha e^{-x}dx\\
&=&\int_0^\infty\bigg[\frac{2k+\alpha+1}{k+1}L_k^\alpha(x)-\frac{x}{k+1}L_k^\alpha(x)-\frac{k+\alpha}{k+1}L_{k-1}^\alpha(x)\bigg]^2 x^\alpha e^{-x}dx\ \ \ \  \text{(Recursive \ Relation)}\nonumber\\
&=&\bigg[\frac{(2k+\alpha+1)}{(k+1)}\bigg]^2 \frac{\Gamma(k+\alpha+1)}{\Gamma(k+1)}
+\frac{1}{(k+1)^2}\bigg[\frac{\Gamma(k+\alpha+3)}{\Gamma(k+1)}+4\frac{\Gamma(k+\alpha+2)}{\Gamma(k)}+\frac{\Gamma(k+\alpha+1)}{\Gamma(k-1)}\bigg] \nonumber\\
&&+\bigg[\frac{(k+\alpha)}{(k+1)}\bigg]^2\frac{\Gamma(k+\alpha)}{\Gamma(k)}
-\frac{2(2k+\alpha+1)}{(k+1)^2}\bigg[\frac{\Gamma(k+\alpha+2)}{\Gamma(k+1)}+\frac{\Gamma(k+\alpha+1)}{\Gamma(k)}\bigg]
-\frac{2(k+\alpha)}{(k+1)^2}\frac{\Gamma(k+\alpha+1)}{\Gamma(k)}\nonumber\\
&=&\frac{1}{(k+1)^2}\bigg[(2k+\alpha+1)^2\frac{\Gamma(k+\alpha+1)}{\Gamma(k+1)}+\frac{\Gamma(k+\alpha+3)}{\Gamma(k+1)}+\frac{4\Gamma(k+\alpha+2)}{\Gamma(k)}-(4k+2\alpha+2)\frac{\Gamma(k+\alpha+2)}{\Gamma(k+1)}\nonumber\\
&&+\frac{\Gamma(k+\alpha+1)}{\Gamma(k)}(-3-4k-3\alpha)\bigg]\nonumber\\
&=&\frac{1}{(k+1)^2}\bigg[(2k+\alpha+1)(-1-\alpha)\frac{\Gamma(k+\alpha+1)}{\Gamma(k+1)}+\frac{\Gamma(k+\alpha+3)}{\Gamma(k+1)}+(1+\alpha)\frac{\Gamma(k+\alpha+1)}{\Gamma(k)}\bigg]\nonumber\\
&=&\frac{1}{(k+1)^2} \bigg[\frac{\Gamma(k+\alpha+3)}{\Gamma(k+1)}-(\alpha+1)\frac{\Gamma(k+\alpha+2)}{\Gamma(k+1)}\bigg]\nonumber\\
&=&\frac{\Gamma(k+\alpha+2)}{\Gamma(k+2)}.\nonumber
 \end{eqnarray}
Therefore, \eqref{app:eq0} holds for all $n$.
\end{proof}

\makeatother
\setcounter{equation}{0}
\renewcommand{\theequation}{B. \arabic{equation}}
\section{Verification of some properties of the pair $\{X_n, Y_n\}$}
Recall that
\begin{eqnarray}
&&X_n=\text{span}\bigg\{\frac{\Gamma(n+\lambda+1)}{\Gamma(n-1+\lambda+\nu)}\hat{L}_n^{\lambda+\nu-2}(x)\bigg\}\nonumber\\
 &&Y_n=\text{span}\{\hat{L}_n^\lambda(x)\}
\end{eqnarray}
We define a simple projection
\begin{equation}
\Pi_nx_n=y_n.
\end{equation}
The fact that $Y_n\cap X_n^\bot=\{0\}$ is obvious. We only need to verify that $\{X_n, Y_n\}$ is a regular pair
\begin{eqnarray}
\|x_n\|^2_{\hat{w}^{2-\lambda-\nu}}&=&\bigg(\frac{\Gamma(n+\lambda+1)}{\Gamma(n-1+\lambda+\nu)}\bigg)^2 \frac{\Gamma(n+\lambda+\nu-1)}{\Gamma(n+1)}=\mathcal{O}(n^{2+\lambda-\nu})\nonumber\\
(x_n, y_n)_{\hat{w}^{2-\lambda-\nu}}&=&\frac{\Gamma(n+\lambda+1)}{\Gamma(n-1+\lambda+\nu)}\int_0^\infty x^\lambda e^{-2\sigma x}L_n^{\lambda+\nu-2}(2\sigma x) L_n^\lambda(2\sigma x)dx\nonumber\\
&=&\frac{\Gamma(n+\lambda+1)}{\Gamma(n-1+\lambda+\nu)}\int_0^\infty x^\lambda e^{-2\sigma x}L_n^{\lambda}(2\sigma x) L_n^\lambda(2\sigma x)dx\nonumber\\
&=&\bigg(\frac{1}{2\sigma}\bigg)^{1+\lambda}\frac{\Gamma(n+\lambda+1)^2}{\Gamma(n-1+\lambda+\nu)\Gamma(n+1)}\nonumber\\
&=&\mathcal{O}(n^{2+\lambda-\nu})
 \end{eqnarray}

\begin{eqnarray}\label{appendB:eq:y}
\|y_n\|^2_{\hat{w}^{2-\lambda-\nu}}&=&\int_0^\infty x^{2+\lambda-\nu}e^{-2\sigma x}[L_n^\lambda(2\sigma x)]^2 dx\nonumber\\
&=& \int_0^\infty x^{2+\lambda-\nu}e^{-2\sigma x}\bigg[\sum\limits_{k=0}^n\binom{n-k+\nu-3}{n-k}L_k^{2+\lambda-\nu}(2\sigma x)\bigg]^2 dx\nonumber\\
&=&\bigg(\frac{1}{2\sigma}\bigg)^{1+\lambda} \sum\limits_{k=0}^n\binom{n-k+\nu-3}{n-k}^2\frac{\Gamma(k+3+\lambda-\nu)}{\Gamma(k+1)}\nonumber\\
&=&\bigg(\frac{1}{2\sigma}\bigg)^{1+\lambda} \sum\limits_{k=0}^n \bigg(\frac{(\nu-2)_{n-k}}{(n-k)!}\bigg)^2\frac{\Gamma(k+3+\lambda-\nu)}{\Gamma(k+1)}\nonumber\\
&=&\bigg(\frac{1}{2\sigma}\bigg)^{1+\lambda} \sum\limits_{k=0}^n \bigg(\frac{\Gamma(n-k+\nu-2)}{\Gamma(n-k+1)\Gamma(\nu-2)}\bigg)^2\frac{\Gamma(k+3+\lambda-\nu)}{\Gamma(k+1)}
\end{eqnarray}
It is trivial that
\begin{equation}
\|y_n\|^2_{\hat{w}^{2-\lambda-\nu}}>\bigg(\frac{1}{2\sigma}\bigg)^{1+\lambda} \frac{\Gamma(n+3+\lambda-\nu)}{\Gamma(n+1)}. \end{equation}
From \eqref{appendB:eq:y}, we also have
\begin{eqnarray}
\|y_n\|^2_{\hat{w}^{2-\lambda-\nu}}&<& \frac{\Gamma(n+3+\lambda-\nu)}{\Gamma(n+1)}\bigg(\frac{1}{2\sigma}\bigg)^{1+\lambda} \sum\limits_{k=0}^n \bigg(\frac{\Gamma(n-k+\nu-2)}{\Gamma(n-k+1)\Gamma(\nu-2)}\bigg)^2\nonumber\\
&=&\frac{\Gamma(n+3+\lambda-\nu)}{\Gamma(n+1)}\bigg(\frac{1}{2\sigma}\bigg)^{1+\lambda} \sum\limits_{k=0}^n \bigg(\frac{\Gamma(k+\nu-2)}{\Gamma(k+1)\Gamma(\nu-2)}\bigg)^2\nonumber\\
&=&\frac{\Gamma(n+3+\lambda-\nu)}{\Gamma(n+1)}\bigg(C+\sum\limits_{k=3}^n \bigg[\frac{1}{k(k-1)}\bigg]^2\bigg[\frac{\Gamma(k+\nu-2)}{\Gamma(k-1)}\bigg]^2\bigg)
\end{eqnarray}
Since $\Gamma(x)$ is an increasing function for $x\geq 2$\cite[Page 255]{AS} and $k+\nu-2<k-1$, we clearly have
$\Gamma(k+\nu-2)/\Gamma(k-1)<1$. Hence,
\begin{equation}
\|y_n\|^2_{\hat{w}^{2-\lambda-\nu}}\leq C\frac{\Gamma(n+3+\lambda-\nu)}{\Gamma(n+1)}.
\end{equation}
Therefore, $\|y_n\|^2_{\hat{w}^{2-\lambda-\nu}}=\mathcal{O}(n^{2+\lambda-\nu})$ and $\{X_n, Y_n\}$ is a regular pair.

\end{document}